\documentclass{article}
\usepackage[left=2cm,right=2cm,top=2cm,bottom=2cm]{geometry} 
\usepackage{amsmath}
\usepackage{microtype}
\usepackage{graphicx}
\usepackage[caption=false,font=footnotesize]{subfig}
\usepackage{booktabs} 
\usepackage{cite}
\usepackage{amsmath,amssymb,xcolor,tikz}
\usetikzlibrary{shapes.geometric, arrows.meta, positioning}
\usepackage{amsfonts}
\usepackage{amsthm,comment,enumitem}
\usepackage{algorithm} 
\usepackage{multicol,multirow,diagbox,hyperref,algorithmic}
\usepackage{authblk,appendix}
\usepackage{bm}
\def\D{\bm{D}}
\def\X{\bm{X}}
\def\Pb{\bm{P}}
\def\p{\bm{p}}
\def\real{\mathbb{R}}
\def\rank{\mathrm{rank}}
\def \alphab {\bm{\alpha}}
\def \betab {\bm{\beta}}
\def \v{\bm{v}}

\def \alphab {\bm{\alpha}}
\def \betab {\bm{\beta}}
\def \e{\bm{e}}

\def\mf{\mathcal{M}}
\def\mfr{\mathcal{M}_r}
\def\diag{\mathrm{diag}}

\def\A{\bm{A}}
\def\B{\bm{B}}

\def\E{\bm{E}}

\def\Ka{\mathcal{K}_\alpha}
\def\G{\bm{G}}
\def\H{\bm{H}}
\def\I{\bm{I}}
\def\J{\bm{J}}

\def\cO{\mathcal{O}}
\def\R{\bm{R}}

\def\M{\bm{M}}

\def\U{\bm{U}}
\def\W{\bm{W}}
\def\X{\bm{X}}
\def\Y{\bm{Y}}
\def\Z{\bm{Z}}
\def\P{\bm{P}}

\def\a{\bm{a}}
\def\ai{\bm{\alpha}\in\mathbb{I}}
\def\b{\bm{b}}

\def\u{\bm{u}}

\def\x{\bm{x}}
\def\y{\bm{y}}
\def\T{\mathbb{T}}
\def\Ta{\mathcal{T}_\alpha}

\def\univ{\mathbb{I}}
\def\oot{\bm{1}\bm{1}^\top}
\def\one{\bm{1}}

\def\Rnn{\mathbb{R}^{n\times n}}
\def\Ro{\mathcal{R}_{\Omega}}
\def\Rol{\mathcal{R}_{\Omega_{l+1}}}
\def\Po{\mathcal{P}_{\Omega}}

\def\Fo{\mathcal{F}_\Omega}
\def\fro{\mathrm{F}}
\def\Pt{\mathcal{P}_{\mathbb{T}}}
\def\Ptl{\mathcal{P}_{\mathbb{T}_l}}
\def\Pthl{\mathcal{P}_{\hat{\mathbb{T}}_l}}
\def\Pu{\mathcal{P}_U}
\def\Pup{\mathcal{P}_{U^\perp}}
\def\Pul{\mathcal{P}_{U_l}}
\def\Puhl{\mathcal{P}_{\hat{U}_l}}
\def\wa{\bm{w}_{\bm{\alpha}}}

\def\wb{\bm{w}_{\bm{\beta}}}

\def\va{\bm{v}_{\bm{\alpha}}}

\def \alphab {\bm{\alpha}}
\def \betab {\bm{\beta}}

\newcommand{\bb}{\mathbb}

\renewcommand{\le}{\leqslant}

\DeclareMathOperator{\Tr}{Tr}

\renewcommand{\epsilon}{\varepsilon}

\newcommand{\la}{\langle}
\newcommand{\ra}{\rangle}

\newtheorem{thm}{Theorem}[section]

\newtheorem{lem}[thm]{Lemma}
\newtheorem{asp}[thm]{Assumption}
\newtheorem{remark}{Remark}

\DeclareMathOperator*{\minimize}{\mathrm{minimize}}
\DeclareMathOperator*{\subjectto}{\mathrm{subject~to}}
\begin{document}
\title{Riemannian Optimization for Non-convex Euclidean Distance Geometry with Global Recovery Guarantees}

\author[1]{Chandler Smith\thanks{Corresponding Author, Chandler.Smith@Tufts.edu}}
\author[2]{HanQin Cai}
\author[1]{Abiy Tasissa}
\affil[1]{Department of Mathematics, Tufts University, Medford, MA 02155, USA.}
\affil[2]{Department of Statistics and Data Science and Department of Computer Science, University of Central Florida, Orlando, FL 32816, USA.}

\maketitle
\begin{abstract}
 The problem of determining the configuration of points from partial distance information, known as the Euclidean Distance Geometry (EDG) problem, is fundamental to many tasks in the applied sciences. In this paper, we propose two algorithms grounded in the Riemannian optimization framework to address the EDG problem. Our approach formulates the problem as a low-rank matrix completion task over the Gram matrix, using partial measurements represented as expansion coefficients of the Gram matrix in a non-orthogonal basis. For the first algorithm, under a uniform sampling with replacement model for the observed distance entries, we demonstrate that, with high probability, a Riemannian gradient-like algorithm on the manifold of rank-$r$ matrices converges linearly to the true solution, given initialization via a one-step hard thresholding. This holds provided the number of samples, $m$, satisfies $m \geq \cO(n^{7/4}r^2 \log(n))$. With a more refined initialization, achieved through resampled Riemannian gradient-like descent, we further improve this bound to $m \geq \cO(nr^2 \log(n))$. Our analysis for the first algorithm leverages a non-self-adjoint operator and depends on deriving eigenvalue bounds for an inner product matrix of  restricted basis matrices, leveraging sparsity properties for tighter guarantees than previously established. The second algorithm introduces a self-adjoint surrogate for the sampling operator. This algorithm demonstrates strong numerical performance on both synthetic and real data. Furthermore, we show that optimizing over manifolds of higher-than-rank-$r$ matrices yields superior numerical results, consistent with recent literature on overparameterization in the EDG problem.
\end{abstract}


\section{Introduction} \label{sec: Introduction}

The rapid advancement of technology across various scientific fields has greatly simplified data collection. In many practical applications, however, there are limitations to measurements that can lead to incomplete data.  This can be caused by geographic, climatic, or other factors that determine whether a measurement between two points can be obtained, and as such some data may be missing\cite{aldibaja2016improving,marti2013multi}. For instance, in protein structure prediction, nuclear magnetic resonance (NMR) spectroscopy experiments yield spectra for protons that are close together, resulting in incomplete known distance information\cite{clore1993exploring}. Similarly, in sensor networks, we may have mobile nodes with known distances only from fixed anchors \cite{boukerche2007localization,kuriakose2014review}. In these and other scenarios, the fundamental problem is determining the configuration of points based on partial information about inter-point distances. This problem is known as the Euclidean distance geometry (EDG) problem, which has numerous applications throughout the applied sciences \cite{biswas2006semidefinite,ding2010sensor,rojas2012distance,porta2018distance,tenenbaum2000global,glunt1993molecular,trosset1997applications,fang2013using,liberti2008branch,einav2023quantitatively}.  

To formulate this problem mathematically, some notation is in order. Let $\{\p_i\}_{i=1}^n\subset\mathbb{R}^r$ denote a set of $n$ points in $\real^{r}$. We define the $r\times n$ matrix $\Pb =[\p_1,\p_2,...,\p_n]$, which has the points as columns. There are two essential mathematical objects related to $\P$. The first object is the Gram matrix $\X \in \mathbb{R}^{n \times n}$, defined as $\X = \Pb^\top \Pb$. By construction, $\X$ is symmetric and positive semi-definite.
The second object is the squared distance matrix $\D \in \mathbb{R}^{n \times n}$, defined entry-wise as $D_{ij} = \Vert \p_i - \p_j \Vert^2_2$. The reason for working with the squared distance matrix instead of the distance matrix will become clear later. Computing $\D$ given $\P$ is conceptually straightforward. However, the inverse problem of determining $\Pb$ from $\D$ is not immediately straightforward. To address this problem, we need to precisely define what it means to identify $\P$. Since rigid motions and translations preserve distances, there is no unique $\P$ corresponding to a given squared distance matrix $\D$. From here on, we assume the points are centered at the origin, i.e., for $\one$ as a column vector of ones, $\Pb \bm{1} = \bm{0}$. This implies that $\X \bm{1} = \Pb^\top \Pb \bm{1} = \bm{0}$. We refer to $\Pb$ and $\X$ with this property as centered point and centered Gram matrix, respectively. Since the Gram matrix is invariant under rigid motions, these assumptions allow for a one-to-one correspondence between $\D$ and $\X$.

When we have access to all the distances, a central result in \cite{torgerson1952multidimensional} provides the following one-to-one correspondence between $\D$ and a centered $\X$:
\begin{align}
    \X &= -\frac{1}{2}\J\D\J \label{eqn: D to X},  \\ \
    \D &= \diag(\X)\one^\top +\one\diag(\X)^\top - 2\X, \label{eqn: X to D}
\end{align}
where $\diag(\cdot)$ inputs an $n\times n$ matrix and returns a column vector with the entries along the diagonal, and $\J = \I - \frac{1}{n}\oot$. Once $\X$ is reconstructed using the above formula, $\Pb$ can be computed from the $r$-truncated eigendecomposition of $\X$. It is important to note that, as previously mentioned, $\P$ is unique up to rigid motions. This procedure for computing $\Pb$ from a full squared distance matrix $\D$ is known as classical multidimensional scaling (Classical MDS) \cite{young1938discussion,torgerson1952multidimensional,torgerson1958theory,gower1966some}.

In many practical scenarios, the distance matrix may be incomplete, making classical MDS inapplicable for determining the point configuration. However, notice that $\rank(\X) \leq r$, and one can show that $\rank(\D)\leq r+2$ \cite{dokmanic2015euclidean}. This implies that when $r\ll n$, which is often the case in practice, $\X$ and $\D$ are low-rank. This allows us to utilize a rich library of tools from low-rank matrix completion. With that, one technique is to directly apply matrix completion techniques on $\D$\cite{moreira2017novel}. Let $\Omega \subset \{(i,j) \mid 1 \leq i < j \leq n\}$ denote the set of sampled indices corresponding to the strictly upper-triangular part of the distance matrix. Note that, since a distance matrix is hollow and symmetric, it suffices to consider the samples in the upper-triangular part; that is, if $D_{ij}$ is sampled, $D_{ji}$ is also assumed to be sampled. A matrix completion approach would consider the following optimization program to recover $\D$:
\begin{equation}
\begin{split}
\label{eq: Convex MC objective}
    \minimize_{\Z\in \real^{n\times n}}\quad & 
||\Z||_{*}\\
  \subjectto\quad & Z_{ij} = D_{ij} \quad \forall (i,j)\in \Omega,
	\end{split}
\end{equation}
where $||\cdot ||_{*}$ denotes the nuclear norm, which serves as a convex surrogate for rank\cite{fazel2001rank}. The main idea of these tools is that, under some assumptions, the nuclear norm minimization program reconstructs the true low-rank squared distance matrix exactly with high probability from $\cO(nr\log^2(n))$ randomly sampled entries \cite{candes2005decoding,candes2006robust,candes2009exact,recht2010guaranteed,gross2010note}. Another set of techniques \cite{tasissa2018exact,lai2017solve} focus on recovering the point configuration by using the Gram matrix as an optimization variable, and using only partial information from the entries in $\D$. Specifically, these works consider the following optimization program for the EDG problem
\begin{equation}\label{eq:edg_convex}
\begin{split}
 \minimize_{\X\in \real^{n\times n},\,\X = \X^\top,\, \X\succeq \bm{0},\X\bm{1}=\bm{0}}& \quad  ||\X||_{*} \\
 \subjectto \quad \quad \quad~ & \quad X_{ii}+X_{jj}-2X_{ij} = D_{ij} \quad \forall(i,j)\in \Omega,
\end{split}
\end{equation}
where the constraints follow from the relation of $\X$ and $\D$ in \eqref{eqn: X to D} and \eqref{eqn: D to X}.  Due to the challenge of working with the constraints imposed by distance matrices, i.e., an entrywise triangle inequality that must be satisfied in order to remain a distance matrix, this work will follow the latter approach of optimizing over the Gram matrix. We note that, in contrast to completing the square distance matrix $\D$ which has rank at most $r+2$, employing a minimization approach based on a Gram matrix that has rank at most $r$ implicitly enforces the constraints of the Euclidean distances. Recent works have indicated that this approach can achieve better sampling complexity than direct distance matrix completion\cite{tasissa2018exact,lai2017solve,Li2024}.

We note that theoretical guarantees for \eqref{eq:edg_convex} have been established in \cite{tasissa2021low,tasissa2018exact}, but still suffer from the lack of scalability of convex techniques. A non-convex Lagrangian formulation was also proposed in \cite{tasissa2018exact}, yielding strong numerical results but lacking local convergence guarantees. The work in \cite{Nguyen2019} uses a Riemannian manifold approach to develop a conjugate gradient algorithm for estimating the underlying Gram matrix. The theoretical analysis therein shows that the squared distance matrix iterates globally converge to the true squared distance matrix at the sampled entries under three assumptions. However, the relationship between the problem parameters, such as the sampling scheme and sampled entries, and the third assumption remains unclear, as noted in Remark III.8 of the paper. In \cite{Li2024}, the authors introduce a Riemannian conjugate gradient method with line search for the EDG problem. The paper provides a local convergence analysis for the case where the entries of the distance matrix are sampled according to the Bernoulli model given a suitable initialization. The initialization method used is known as rank reduction, which begins with initial points embedded in a higher-dimensional space than the target dimension. While \cite{Li2024} demonstrates strong empirical results for this initialization via tests on synthetic data for sensor localization, there are no provable guarantees provided for the initialization. In this manuscript, we aim to present a provable non-convex algorithm for the EDG problem, along with a provable initialization. 

\subsection{Contributions}
The main contributions of this paper are as follows:

\begin{enumerate} [leftmargin=*]
    \item {\bf Construction of two novel algorithms:} We propose two non-convex iterative algorithms in a Riemannian optimization framework for the Euclidean distance geometry problem. These algorithms are both smooth, first-order methods on the manifold of rank-$r$ matrices for a fixed $r$, and have low computational complexity per iteration. 
    \item {\bf Two different initialization methods:} We propose two different structured initializations from partial measurements, and prove an error bound between the initializations and true solution. Both initializations are relatively simple and require minimal a priori knowledge of the ground truth matrix, save from the measurements necessary to construct the algorithm.
    \item {\bf Convergence guarantees and sample complexity requirements:} We provide theoretical analysis that ensures high probability local convergence of one of the algorithms to the ground truth solution. Along with this characterization of an attractive basin, we prove sample complexity results for the initialization methods to guarantee the algorithm's starting point within the attractive basin.
\end{enumerate}

\subsection{Notation}
We briefly summarize the notation used throughout this paper below. In general, uppercase boldface scripts, such as $\A$, will denote matrices, lowercase boldface scripts, such as $\v$, will denote vectors, calligraphic scripts, such as $\mathcal{A}$, will denote linear operators on matrices, and blackboard bold font, such as $\mathbb{V}$, will denote vector spaces and subspaces. $\X^\top$ denotes the transpose of $\X$, $\Tr(\X)$ as the trace of $\X$, $\la \A,\B\ra = \Tr(\A^\top\B)$ denotes the trace inner product, and $\delta_{ij}$ denotes the Kronecker delta. We denote the $(i,j)$-th entry of a matrix $\X$ by $X_{ij}$. By $\one$, we mean this to be a column vector of ones, of a size determined by the context, and by $\bm{0}$ we mean either a column vector of zeros or a matrix of zeros. $\e_i$ denotes a vector of zeros except a $1$ at the $i$-th position. We denote $\Vert \x\Vert_2$ to be the standard $l_2$ norm on $\real^n$, $\Vert\X\Vert_\fro$ to be the Frobenius norm on $\Rnn$, $\Vert \X\Vert$ to be the operator norm of a matrix, $\Vert \X\Vert_\infty$ to be the maximum element of $\X$, and $\Vert \X\Vert_\star$ to be the nuclear norm of $\X$. We denote $\Vert\mathcal{A}\Vert = \sup_{\Vert\X\Vert_\fro = 1}\Vert\mathcal{A}(\X)\Vert_\fro$ to be the operator norm of linear operators on matrices, and $\lambda_\mathrm{max}(\X)/\lambda_\mathrm{min}(\X)$ to the maximal/minimal eigenvalue of a matrix. We denote $\odot$ as the Hadamard product between two matrices. We denote the $i$-th row of a matrix $\X$ by $\X^{(i)}$, and the $i$-th column by $\X_{(i)}$. We denote the universal set of indices as $\mathbb{I}$ and random subsets of $\mathbb{I}$ by $\Omega$. We denote the empty set as $\emptyset$. We denote the standard matrix basis as $\{\e_{ij}\}_{i,j = 1}^n$, where $\e_{ij} = \e_i\e_j^\top$, which is zero everywhere except a $1$ in the $(i,j)$-th entry. We denote the map $\Vec(\cdot)$ as the operation that takes in a matrix $\Y\in\Rnn$ and returns a column vector, with each column of $\Y$ stacked in order, in $\real^{n^2}$. We define the thin spectral decomposition of a symmetric rank-$r$ matrix as $\Y = \U\D\U^\top$, where $\U\in\real^{n\times r}$ and $\D\in\real^{r\times r}$. We define $\mathcal{I}$ as the identity operator on matrices, and $\I$ as the identity matrix. We denote the condition number $\kappa$ of a rank-$r$ matrix $\Y$ as $\kappa = \frac{\Vert \Y\Vert}{\sigma_r(\Y)}$, where $\sigma_r(\Y)$ is the smallest non-zero singular value.

We denote the manifold of rank-$r$ matrices as $\mfr$, and general smooth manifolds as $\mf$. We denote the tangent space of the ground truth solution $\X\in\mfr$ to be $\T$, and the tangent space of the $l$-th iterate in the iterative sequences defined in Section~\ref{sec: Analysis results} as $\T_l$. We denote the Euclidean gradient of a function $f\in C^1(\real^n)$ as $\nabla f$, and the Riemannian gradient of a function $f\in C^1(\mf)$ as $\mathrm{grad}\,f$.

\subsection{Organization}
The organization of this paper is as follows. In Section~\ref{sec: background}, we discuss the requisite background information necessary to understand the work done in this paper. This consists of a brief discussion of dual bases of a vector space, first-order retraction-based Riemannian optimization methods, low-rank matrix completion, and discussion of EDG. Section~\ref{sec: Related work} discusses related geometric approaches in matrix completion, relevant work done in EDG, and a more detailed discussion of geometric approaches to EDG. Section~\ref{sec: R_omega alg section} is a discussion of our two proposed methodologies for solving the EDG problem using geometric low-rank matrix completion ideas in the developed dual basis framework. Section~\ref{sec: Analysis results} discusses the underlying assumptions, convergence analysis, and initialization guarantees of one of the proposed algorithms, with most proofs deferred to the Appendices. The convergence analysis leverages the discussed dual basis structure, with properties proven in Appendix~\ref{appendix: dual basis}, to get local convergence guarantees, discussed in more detail in Appendix~\ref{appendix: local convergence}. We additionally provide initialization guarantees in this section, with relevant proofs in Appendix~\ref{appendix: initialization}. Section~\ref{sec: Numerics} discusses the numerical results of these algorithms. We conclude the paper in Section~\ref{sec: Conclusion} with a brief discussion of the work and possible future research directions.

\section{Background}\label{sec: background}
In this section, we will provide some minor background necessary to understand the work done in the following sections.

\subsection{Dual Basis}
In a finite dimensional vector space of matrices $\mathbb{V}$, where $\mathrm{dim}(\mathbb{V}) = n$, a basis is a linearly independent set of matrices $B = \{\X_i\}_{i=1}^{n}$ that spans $\mathbb{V}$. Any basis for a finite dimensional vector space admits a dual, or bi-orthogonal, basis denoted $B^{*} = \{\Y_i\}_{i=1}^{n}$ that also spans $\mathbb{V}$, and admits a bi-orthogonality relationship
\[
\langle \X_i,\Y_j \ra = \delta_{ij}.
\]
Additionally, $B$ uniquely determines $B^*$. The bi-orthogonality relationship allows for the decomposition of any matrix $\Z\in \mathbb{V}$ as follows:
\[
\Z = \sum_{i=1}^{n} \langle \Z,\Y_i\rangle \X_i = \sum_{i=1}^{n} \langle \Z,\X_i\rangle \Y_i. 
\]
We define the Gram, or correlation matrix, $\H\in\Rnn$, for $B$ as $H_{ij} = \la\X_i,\X_j\ra$, and let $H^{ij} = (\H^{-1})_{ij}$. It is straightforward to show that $\Y_i = \sum_{j=1}^{n} H^{ij}\X_j$ generates $B^*$, and similarly that $\X_i = \sum_{j=1}^{n}H_{ij}\Y_j$ \cite{bhatia2013matrix}.

\subsection{Riemannian Optimization}\label{subsec: intro to RO}

The primary setting for this work is the Riemannian manifold of fixed-rank matrices. Throughout this work, we will only be considering square $n \times n$ matrices for simplicity and relevance to the problem of interest in this paper. For a fixed positive integer $r\leq n$, we denote the set $\mfr = \{\X\in\Rnn~\vert ~\mathrm{rank}(\X) = r\}$. Although not obvious at first glance, it is well-known that $\mfr$ is a smooth Riemannian manifold\cite{vandereycken2012lowrank,Boumal2023}. To make this a Riemannian manifold, we equip it with the standard trace inner product as a metric, or $\langle\A,\B\rangle = \Tr(\A^\top\B)$, restricted to the tangent bundle $T\mfr$, which is the disjoint union of tangent spaces\cite{Boumal2023}.

Additionally, the tangent space at a point $\X\in\mfr$ is known and can be characterized \cite{vandereycken2012lowrank,Boumal2023,wei2020guarantees}. For notational simplicity, and of relevance in the context of optimization, assume that $\X$ is the ground truth solution to an objective function. We additionally assume that $\X = \X^\top$, as all the matrices we consider are symmetric. The following ideas can be re-stated for rectangular matrices using a singular value decomposition, but these are not the subject of this paper. As such, we denote the tangent space at $\X$ as $\T$, and for a sequence of iterates $\{\X_l\}_{l\geq0}$, we refer to their respective tangent spaces as $\T_l$. To characterize $\T$, let $\X = \U\D\U^\top$ be the thin spectral decomposition of $\X$. The tangent space $\T$ can be computed as follows:
\[
\T = \{\U \Z^\top + \Z\U^\top ~\vert~ \Z\in\real^{n\times r}\}.
\]
The tangent space can be described as the set of all possible rank-up-to-$2r$ perturbations, represented as the sum of a perturbation in the column and row space, and is computed by looking at first-order perturbations of the spectral decomposition of $\X$\cite{vandereycken2012lowrank}. Additionally, we can compute the orthogonal projection of an arbitrary $\Y\in\Rnn$ onto the tangent space at a point $T_{\X} \mfr$ as follows \cite{vandereycken2012lowrank,Boumal2023,wei2020guarantees}:
\begin{equation*}
    \Pt \Y  = \Pu\Y + \Y\Pu - \Pu\Y\Pu
\end{equation*}
where $\Pu = \U\U^\top$ is the orthogonal projection onto the subspace spanned by the $r$ columns of $\U$.

Optimization over $\mfr$ has been investigated in detail for quite some time, and in particular retraction-based methods are of particular interest to this work \cite{Absil2008,shalit2012,vandereycken2012lowrank,wei2016guarantees,cai2019accaltproj,wei2020guarantees,cai2021asap,hamm2022RieCUR}. First-order retraction-based methodologies rely on the general principle of taking a descent step in the tangent space, followed by a retraction onto the manifold. In the case of first-order optimization on $\mfr$, the retraction map $\mathcal{H}_r$ is given by the hard thresholding operator, which is a thin spectral decomposition that takes $\Y=\sum_{i=1}^n\lambda_i\u_i\u_i^\top\mapsto\sum_{i=1}^r \lambda_i\u_i\u_i^\top$, where $\vert\lambda_1\vert\geq ...\geq \vert\lambda_n\vert$ are the ordered eigenvalues of $\Y$ and $\u_i$ are the corresponding eigenvectors of $\Y$.

In order to construct a first-order method on $\mfr$, we need to define the notion of a Riemannian gradient. This object can be constructed in a greater degree of generality than our approach, but for simplicity, we will assume that a function $f:\mfr\to \real$ can be smoothly extended to all of $\Rnn$. That is to say, if we consider $f:\Rnn\to\real$, the Riemannian gradient of $f\big\vert_{\mfr}$, denoted $\mathrm{grad} ~f$, for $\X_l\in\mfr$ is given by:
\[
\mathrm{grad}\,f(\X_l) = \Ptl \nabla f(\X_l),
\]
where $\nabla f$ is the Euclidean gradient of $f$. Using this approach, we can now define a Riemannian gradient descent iterate sequence using our retraction map, Riemannian gradient, and some step size sequence $\{\alpha_l\}_{l\geq 0}$ as follows:
\begin{equation}\label{eqn: first order retraction}
\X_{l+1} = \mathcal{H}_r(\X_l-\alpha_l \Ptl \nabla f(\X_l)).
\end{equation}
Intuitively, this algorithm seeks to look at changes in the objective function that lie, locally, along the manifold, followed by a retraction to stay on the desired manifold. An illustration can be seen in Figure~\ref{fig:First order retraction figure}.

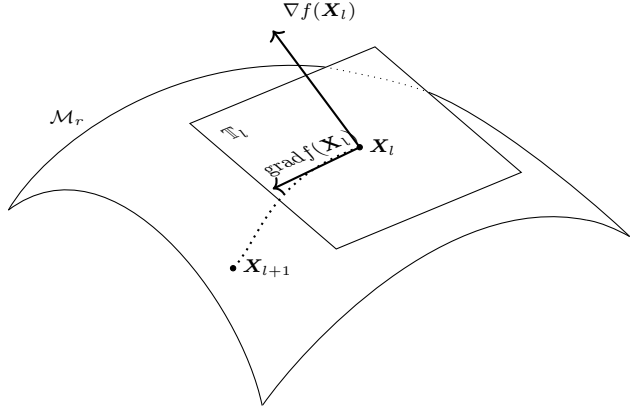
\begin{figure}
    \centering
\begin{tikzpicture}[x=1pt, y=1pt, yscale=-1]

  height of 300

  \draw (112.04,66.62) node[below right, font=\scriptsize]{$\mfr$};
  \draw (100,108) .. controls (140,78) and (188,143) 
                  .. (196,182) ;
  \draw (196.2,181.8) .. controls (211.8,159.6) and (265.71,101.03) 
                      .. (318.8,111.8) ;
  \draw (318.8,111.8) .. controls (323.67,112.42) and (330.33,115.42) 
                      .. (335,118) ;
  \draw[dotted] (219.48,53.8) .. controls (230.71,56.06) and (244.01,58.07) 
                              .. (259,63.39) ;
  \draw (259,63.39) .. controls (288.81,75.27) and (332.14,114.93) 
                    .. (335,118) ;
  \draw (100,108) .. controls (113.56,84.56) and (160.04,48.06) 
                  .. (219.48,53.8) ;

  \filldraw (232.9,84.2) circle (1pt)
                         node[right, font=\scriptsize]{$\X_l$};

  \draw[thick,->] (232.9,84.2) --  (200,40)
                    node[above right,font=\scriptsize]{$\nabla f(\X_l)$};

  \draw[dotted,thick] (232.9,84.2) .. controls (198,96) and (195,120) 
                .. (185,130);

  \draw[thick, ->] (232.9,84.2) -- (200,100)
                        node[above right,font=\scriptsize,rotate=-337.86,
                           xslant=-0.54]{$\textrm{grad}f(\X_l)$};

    \filldraw (185,130)circle (1pt)
                         node[right, font=\scriptsize]{$\X_{l+1}$};                       

  \draw (238.68,46.2) -- (294.01,93.72) 
                      -- (224.01,122.72) 
                      -- (168.68,75.2) 
                      -- cycle ;
  \draw (173.51,76.9) node[below right, 
                           font=\scriptsize, 
                           rotate=-337.86,
                           xslant=-0.54]{$\T_l$};
\end{tikzpicture}

    \caption{A diagram of a simple first-order retraction method on $\mfr$. Again, $\nabla f(\X_l)$ is the Euclidean gradient of $f$ at $\X_l$, $\mathrm{grad}\, f(\X_l)$ is the Riemannian gradient at $\X_l$, and $\X_{l+1} = \mathcal{H}_r(\X_l - \alpha_l\mathrm{grad}\,f(\X_l))$, as in \eqref{eqn: first order retraction}.}
    \label{fig:First order retraction figure}
\end{figure}

This is a simple first pass to first-order optimization on Riemannian manifolds, and is not meant to be exhaustive. Interested readers should consult \cite{Absil2008,Boumal2023} for further details on first-order methods on matrix (and other Riemannian) manifolds, along with convergence analysis for these algorithms.

\subsection{Matrix Completion}
One of the primary components this work relies on is the field of low-rank matrix completion, where a subset of the entries of a low-rank ground truth matrix $\X$ are observed. Consider $\X$ as an $n \times n$ matrix for simplicity, with $\Omega\subset [n]\times[n]$ representing the set of observed indices. Here, a sampling operator $\Po:\Rnn\to\Rnn$ is introduced, which aggregates the observed entries of $\X$ projected onto specific basis elements $\e_{ij}$:
\begin{equation}\label{eqn: Po definition}
    \Po(\X) = \sum_{(i,j)\in\Omega}\langle \X,\e_{ij}\rangle \e_{ij}.
\end{equation}
If $\Omega$ does not contain any repeated indices, $\Po$ is an orthogonal projection operator. The standard low-rank matrix completion problem can be phrased as follows:
\begin{equation*}
    \minimize_{\Y\in\Rnn} ~\rank(\Y) ~ \subjectto ~ \Po(\Y) = \Po(\X).
\end{equation*}
As minimizing the rank directly is generally a challenging problem\cite{candes2009exact,meka2008rank}, relaxations of this problem are often considered. For details on complexity class of rank constrained problems, we refer the reader to \cite{bertsimas2022mixed}. 
Exact recovery of $\X$ from $\Po(\X)$ using a convex relaxation to the nuclear norm, such as the objective described in \eqref{eq: Convex MC objective},  is a well-studied problem \cite{candes2006robust,recht2011simpler,gross2011recovering} with strong convergence guarantees. This problem is at the core of matrix completion literature, and has inspired work in the completion of distance matrices \cite{lai2017solve,tasissa2018exact}. However, solving the convex problem is expensive for large matrices, which has led to the consideration of non-convex methodologies to solve the underlying problem. One approach that has received a great deal of attention is the Burer-Monteiro factorization approach, pioneered for semi-definite methods in \cite{burer2003nonlinear}, whereby a low rank matrix $\X\in\Rnn$ can be factored into a product $\X = \A\B^\top$ for $\A,\B\in\real^{n\times r}$. Minimizing $\Vert \Po(\X) - \Po(\A\B^\top)\Vert_\fro^2$ is a common approach, and is often dealt with using alternating minimization methods in both the noiseless and noisy case \cite{jain2013low,Hardt2014,zhang2016provable,Chen2020}.
\\

\subsection{Dual Basis Approach to EDG}\label{subsec: Dual Basis EDG}

In the EDG problem, using the relation \eqref{eqn: X to D}, we can relate each entry of the squared distance matrix to the Gram matrix
as follows: $D_{ij} = X_{ii} + X_{jj} - X_{ij}-X_{ji}$. We describe here briefly the dual
basis approach introduced in \cite{tasissa2018exact}. Given $\alphab = (\alpha_1, \alpha_2), \alpha_1 < \alpha_2$,
we define the matrix $\wa$ as follows:
\begin{align*}
    \wa = \e_{\alpha_1\alpha_1}+\e_{\alpha_2\alpha_2}-\e_{\alpha_1\alpha_2}-\e_{\alpha_2\alpha_1}.
 \end{align*}
If we consider the set $\mathbb{I}=\{(\alpha_1,\alpha_2),1\le \alpha_1<\alpha_2\le n\}$, it can be checked that
the set $\{ \wa \}$ is a non-orthogonal basis for the subspace of symmetric matrices with zero row sum, denoted $\mathbb{S} = \{\Y\in\Rnn~\vert~\Y=\Y^\top, \Y\one=\bm{0}\}$. In fact, for any two pairs of indices $\alphab,\betab\in\mathbb{I}$, we have:
\begin{equation*}
    \langle\wa,\wb\rangle = \begin{cases}
        4 & \alphab=\betab;\\
        1 & \alphab\neq\betab,~\alphab\cap\betab\neq\emptyset;\\
        0 & \alphab\cap\betab = \emptyset.
    \end{cases}
\end{equation*}
It can also easily be verified that the dimension of the linear space $\mathbb{S}$ is $L=n(n-1)/2$.
Using this basis, we can realize each entry of the squared distance matrix as the trace inner product of the Gram matrix with the basis. Formally, $D_{ij} = \langle \X, \wa \rangle$ for $\alphab = (i,j)$. Further, we can introduce the dual basis
to $\{ \wa \}$, denoted as $\{ \va \}$, and represent any centered Gram matrix $\X$ using
the following expansion:
\[
\X = \sum_{\alphab} \langle \X\,,\wa\rangle \va.
\]
The advantage of the dual basis representation is that it allows us to recast the EDG problem as a low-rank matrix recovery problem where we observe a subset
of the expansion coefficients. In \cite{tasissa2018exact}, this dual basis formulation has been used to provide theoretical guarantees for the convex program given in \eqref{eq:edg_convex}. 

To make use of the dual basis approach both in theory and applications, one of the first steps is to have a representation
of the dual basis that is easier to use. The direct form of the dual basis, based on its definition, relies on an inverse of a matrix of size $L\times L$ which requires the solution of a large linear system. In \cite{lichtenberg2023dual},
it was shown that the dual basis admits a simple explicit form
\begin{equation}\label{eqn: v_a form}
    \va = -\frac{1}{2}\left(\a\b^\top + \b\a^\top\right),
\end{equation}
where $\a = \e_i - \frac{1}{n}\one$ and $\b = \e_j - \frac{1}{n}\one$ for $\alphab = (i,j)$.
We now highlight
a few operators that are related to the dual basis approach. The first one is the sampling operator
$\Ro:\mathbb{S}\to\mathbb{S}$ defined as follows:
\begin{equation*}
    \Ro(\cdot) = \sum_{\alphab\in\Omega}\langle\cdot,\wa\rangle\va.
\end{equation*}

The bi-orthogonality relationship of the dual basis gives that $\Ro^2 = \Ro$ if $\Omega$ does not have repeated indices, and that
\begin{equation*}
    \Ro^\star(\cdot) = \sum_{\alphab\in\Omega}\la \cdot,\va\ra\wa.
\end{equation*}
Due to the lack of self-adjointness, $\Ro$ without repeated indices in $\Omega$ is not an orthogonal projection operator, and is instead an oblique projection operator. In \cite{Smith2023}, $\Ro(\X)$ is related to the sampling operator $\Po(\D)$ as follows:
\begin{equation}\label{eqn: Ro and Po equiv}
    \Ro(\X) = -\frac{1}{2} \J\Po(\D)\J,
\end{equation}
where $\J$ is as defined in Section~\ref{sec: Introduction}. The next operator is the restricted frame operator $\Fo:\mathbb{S}\to\mathbb{S}$, first studied in \cite{tasissa2018exact}, and defined as
\begin{equation}\label{eqn: F_omega}
    \Fo(\cdot) = \sum_{\alphab\in\Omega}\la\cdot,\wa\ra\wa.
\end{equation}
This operator is self-adjoint, positive semi-definite, but unlike $\Ro$, does not reference the dual basis. We note that this operator under a different name was critical to the analysis of the algorithm in \cite{Li2024}.  

\section{Related Work}\label{sec: Related work}

\subsection{A Riemannian Approach to Matrix Completion}\label{subsec: RO and MC}

A notable non-convex approach is to utilize prior knowledge regarding the rank of $\X$. This methodology centers around the fact that the set of fixed-rank matrices forms a Riemannian manifold, turning the problem into an unconstrained optimization task over a manifold. These methodologies lose convexity, however, and generally only local convergence guarantees can be established, done by proving the existence of attractive basins around solutions. Various retraction-based methodologies have been used with differing metrics and geometric structures\cite{vandereycken2012lowrank,mishra2013fixedrank,Boumal2015,dai2010,keshavan2010matrix,keshavan2010matrix2,wei2020guarantees}. The analysis conducted by \cite{wei2020guarantees} stands out for its interpretation of its first-order method as an iterative hard-thresholding algorithm with subspace projections and efficient numerical implementation. This implementation is done by reducing the hard thresholding step from a thin eigenvalue decomposition of an $n\times n$ matrix to a thin QR decomposition followed by a full eigenvalue decomposition of a far smaller $2r\times 2r$ matrix. The convergence analysis in this work builds on the analysis done in \cite{wei2020guarantees}, and as such, a brief exposition of their work is provided.

 In \cite{wei2020guarantees}, the authors develop a gradient descent algorithm to solve the low-rank matrix completion problem leveraging this Riemannian structure. The objective function used in \cite{wei2020guarantees} is as follows:
\begin{equation}\label{eqn: P_omega qf objective}
    \minimize_{\Y\in\Rnn} ~\langle \Y-\X,\Po(\Y-\X)\rangle ~ \subjectto ~ \rank(\Y) = r.
\end{equation}
The authors used a uniform sampling at random with replacement model for recovering a subset of the indices of the ground truth matrix. This is standard practice in existing matrix completion literature, as much of the analysis relies on concentration inequalities for sums of random matrices to get high probability guarantees. It follows that \eqref{eqn: P_omega qf objective} is not equivalent to $\Vert \Po(\X-\M)\Vert_\fro^2$ when indices in $\Omega$ repeat, as $\Po^2 \neq \Po$ when this occurs. This is distinct from \cite{vandereycken2012lowrank}, which minimized the Frobenius norm difference between the observed entries of the low-rank matrices to solve the problem. Additionally, \cite{vandereycken2012lowrank} demonstrates that the limit of their proposed algorithm agrees with the ground truth in the revealed entries when projected onto the tangent space of the ground truth. However, as the sampling operator has a non-trivial null space, noted in \cite{vandereycken2012lowrank}, this does not necessarily guarantee identification of the ground truth. In contrast, \cite{wei2020guarantees} establishes linear convergence to the ground truth solution in a local neighborhood of the ground truth, with high probability. After defining \eqref{eqn: P_omega qf objective}, \cite{wei2020guarantees} constructs a Riemannian gradient descent procedure similar to the retraction procedure described in Section~\ref{subsec: intro to RO} for its solution.

In addition to this approach, the work in \cite{wei2020guarantees} considered two initialization schemes. One is a simple one-step hard threshold onto $\mfr$, and is given by $\X_0 = \frac{n^2}{m}\mathcal{H}_r(\Po(\M))$. Additionally, a more delicate initialization can be considered by partitioning the set $\Omega$ into $S$ equally sized subsets, and performing one Riemannian gradient descent step for each subset. This Riemannian resampling initialization breaks the dependence on each iterate from the previous, and provides a more reliable initialization for large enough sample sizes. A modification of this technique, applied to our scheme, can be seen in Algorithm~\ref{alg:Resampling Algorithm}.

\subsection{Euclidean Distance Geometry Algorithms}

To solve the EDG problem, various algorithms have been developed. Among them, one prominent family of algorithms is based on semi-definite programming (SDP), which leverages the connection between squared distance matrices and Gram matrices. To provide a concrete example of this approach, we briefly outline the method proposed in \cite{alfakih1999solving}. Consider the matrix $\bm{V} \in \mathbb{R}^{n\times (n-1)}$, whose columns form an orthonormal basis for the space $\{\bm{z} \in \mathbb{R}^{n} : \bm{z}^\top\bm{1} = 0\}$. The operator $\mathcal{K}$ is defined as:
\[
\mathcal{K}(\bm{X}) =  \text{diag}(\bm{X})\bm{1}^\top + \bm{1}\text{diag}(\bm{X})^\top - 2\bm{X}.
\]
This definition of the operator $\mathcal{K}(\bm{X})$ is equivalent to the mapping of the Gram matrix to the squared Euclidean distance matrix, as expressed in \eqref{eqn: X to D}. In \cite{alfakih1999solving}, the optimization program is based on the operator $\mathcal{K}_{\bm{V}}(\bm{X})$, which is defined as $\mathcal{K}_{\bm{V}}(\bm{X}) = \bm{V}\bm{X}\bm{V}^\top$. The optimization problem in \cite{alfakih1999solving} can then formulated as follows:
\begin{equation*}
\begin{split}
 \minimize_{\bm{X} \in \mathbb{R}^{(n-1) \times (n-1)}, \,\, \bm{X} = \bm{X}^\top, \, \bm{X} \succeq \bm{0}} & \quad  \sum_{(i,j) \in \Omega} \left[(\mathcal{K}_{\bm{V}}(\bm{V}\bm{X}\bm{V}^\top))_{ij} - D_{ij}\right]^2.
\end{split}
\end{equation*}
We refer the reader to \cite{alfakih1999solving} for theoretical and numerical aspects of the above optimization program.
Given that standard SDP formulations can be computationally intensive, distributed and divide-and-conquer methods have also been explored. For additional SDP-based formulations of the EDG problem and their applications to molecular conformation and sensor network localization, we refer the reader to \cite{biswas2006semidefinite, biswas2004semidefinite, biswas2008distributed, leung2010sdp, alipanahi2012protein, guo2023perturbation}.

In the context of protein structure determination, various algorithmic approaches to EDG have been developed. One notable example is the EMBED algorithm\cite{havel1991evaluation,more1999distance,crippen1988distance}, which comprises three main steps\cite{havel1998distance}. The first step, known as bound smoothing, involves generating lower and upper bounds for all distances by extrapolating from the available limits of known distances. The second step is the embed step, where distances are sampled from these bounds to form a full distance matrix from which an initial estimate of the protein structure is obtained. The final step involves refining this initial structure by minimizing an energy function using non-convex optimization methods. Another approach to structure prediction is the discretizable molecular distance geometry framework, which can be formulated as a search in a discrete space and then uses a Branch-and-Prune method to estimate the structure \cite{lavor2012recent,lavor2012discretizable}. 

Another category of approaches to the EDG problem involves initially estimating a smaller portion of the point cloud and then using this initial estimate to incrementally reconstruct the rest of of the structure. These methods are referred to as geometric build-up algorithms\cite{wu2007updated,dong2003geometric,sit2009geometric}. The algorithm proposed in \cite{hendrickson1995molecule} addresses the molecular conformation problem by adopting a divide-and-conquer strategy, where a sequence of smaller optimization problems is solved instead of solving a single global optimization problem. Finally, we highlight algorithms that estimate the underlying points through non-convex optimization, utilizing a combination of methods such as majorization, alternating projection, and global continuation (transforming the optimization problem to a function with few local minimizers) \cite{leeuw1977application,
more1997global,glunt1993molecular,fang2012euclidean}. We note that the above discussion does not comprehensively cover all EDG algorithms, and we refer readers to \cite{liberti2014euclidean,dokmanic2015euclidean} for a more detailed overview.

\subsubsection{Related Geometric Approaches to EDG}
The main perspective taken in this paper is in line with low-rank matrix completion approach, albeit not one that employs the trace heuristic seen in \cite{tasissa2018exact,biswas2006semidefinite,Javanmard_2012}. This work is more in line with non-convex approaches based on optimizing over a Riemannian manifold \cite{Nguyen2019,Parhizkar2013}, and extends the Riemannian approach of \cite{wei2020guarantees} to the EDG basis case. 

A recent work in \cite{Li2024} adopts a similar approach to us and considers solving the EDG problem through Riemannian methods as well. In this work, the authors use a Riemannian conjugate method paired with an inexact line search method to minimize the following s-stress objective function:
\begin{equation}\label{eqn: s-stress fn}
    \minimize_{\Y\in\R^{n\times d}} ~\frac{1}{2}\Vert \W\odot \Po(g(\Y\Y^\top)-\D_e)\Vert_\fro^2,
\end{equation}
where $g$ is the map defined by \eqref{eqn: X to D}, $\W$ is a weight matrix to model noisy entries, and $\odot$ is the Hadamard product, and $\Po$ is defined as in \eqref{eqn: Po definition}. The analysis in \cite{Li2024} centers around the minimization of the s-stress function in \eqref{eqn: s-stress fn} using a generalization of a Hager-Zhang line search method to a Riemannian quotient manifold. The main result in this work is that there exists an attractive basin for \eqref{eqn: s-stress fn} that, with high probability, gives linear convergence to the ground truth provided an initialization in the basin. This result requires a Bernoulli sample complexity $p>C\frac{(\nu r)^3log(n)}{n}$ , where $\nu$ is the coherence of the ground truth matrix and $r$ is the rank. Our method also describes two strong initialization candidates for the noiseless EDG recovery problem with provable high probability guarantees, with a sample complexity that only depends quadratically on the coherence and rank.

We provide a separate convergence analysis from \cite{Li2024}, demonstrating a robust Restricted Isometry Property of a non-self-adjoint sampling operator, and prove local convergence for this non-orthogonal matrix completion problem. This novel approach requires a relaxation away from the minimization of a quadratic form over a manifold, instead considering a linearly contractive sequence in a neighborhood modeled after \cite{wei2020guarantees} and a surrogate step size, to be expanded on later. This approach requires novel analysis of the dual-basis framework discussed in \cite{tasissa2021low,tasissa2018exact,lai2017solve}, mainly centered around careful eigenvalue bounds in tandem with standard matrix completion tools, at a cost of slightly worse sample complexity. Additionally, we extend the initialization techniques of \cite{wei2020guarantees}, and show that our modified approach can provide similar guarantees. The non-self adjoint nature of the EDG sampling operator provides a host of challenges that are resolved through careful analysis of the EDG basis and extensions of its properties beyond what has been discovered already. To the authors' knowledge, this is the first non-convex method that provides high probability guarantees on the initialization methods provided.

\

\section{The Riemannian Dual Basis Approach to EDG} \label{sec: R_omega alg section}

With the goal of translating the standard matrix completion problem to Gram matrix completion for EDG in mind, the most direct adaptation of the work conducted in \cite{wei2020guarantees} would be defining an objective function by analogy to \eqref{eqn: P_omega qf objective} as follows:
\begin{equation*}
    \minimize_{\Y\in\mathbb{S}} ~\langle \Y-\X,\Ro(\Y-\X)\rangle ~\subjectto ~\mathrm{rank}(\Y) = r.
\end{equation*}
However, a notable challenge arises: computing the Euclidean gradient of the objective function necessitates unavailable information in the form $\langle \X,\va\rangle$ from $\Ro^\star(\X)$ as
\begin{equation*}
    \nabla_{\Y} \left(\langle \Y-\X,\Ro(\Y-\X)\rangle\right) = \Ro(\Y-\X)+\Ro^\star(\Y-\X),
\end{equation*}
where $\nabla_{\Y}$ denotes the gradient with respect to $\Y$. To circumvent this difficulty, there has been exploration into self-adjoint alternatives to $\Ro$ \cite{Tasissa2021,tasissa2018exact,Smith2023}, one of which we will expand upon shortly. These surrogates allow for the definition of an objective function in analogy to \eqref{eqn: P_omega qf objective}, but as of now lack the requisite theoretical properties for convergence.

The primary surrogate of interest in this work is the restricted frame operator $\Fo$, defined in \eqref{eqn: F_omega}.
This operator is self-adjoint, positive semi-definite, and expresses the same information as $\Ro$ without reference to the dual basis. Additionally, this operator under a different name was critical to the analysis of the algorithm in \cite{Li2024}. Using this operator, we can define the following objective function:
\begin{equation}\label{eqn: F_omega objective function}
    \minimize_{\Y\in\mathbb{S}} ~\frac{1}{2}\langle \Y-\X,\Fo(\Y-\X)\rangle ~\subjectto ~\mathrm{rank}(\Y) = r.
\end{equation}
This is a true quadratic form, minimized over $\mfr$, and can be approached in an identical manner algorithmically as \eqref{eqn: P_omega qf objective}. This operator motivates Algorithm~\ref{alg:F_omega descent}, where the hard thresholding operator $\mathcal{H}_r$ is again defined as the map from $\Y = \sum_{i=1}^n\lambda_i\u_i\u_i^\top \mapsto \sum_{i=1}^r \lambda_i\u_i\u_i^\top$, where $\vert\lambda_1\vert\geq...\geq\vert\lambda_n\vert$:
\begin{algorithm}
\caption{Restricted Frame EDG Riemannian Gradient Descent}\label{alg:F_omega descent}
\begin{algorithmic}
\STATE\textbf{Initialization:} $\X_0 =\U_0\D_0\U_0^\top$
\FOR{$l = 0,1,...$}
    \STATE 1. $\G_l = \Fo(\X-\X_l)$
    \STATE 2. $\alpha_l = \frac{\left\Vert \Ptl\G_l\right\Vert_\fro^2}{\langle \Ptl\G_l,\Fo\Ptl\G_l\rangle }$
    \STATE 3. $\W_l = \X_l + \alpha_l \Ptl\G_l$
    \STATE 4. $\X_{l+1} = \mathcal{H}_r(\W_l)$
\ENDFOR
\STATE\textbf{Output:} $\X_\mathrm{rev}$
\end{algorithmic}
\end{algorithm}

Algorithm~\ref{alg:F_omega descent} is easily implementable and gives strong numerical results, provided in Section~\ref{sec: Numerics}, but proof of local convergence remains an open question. The missing analytical property that would yield local convergence is the Restricted Isometry Property (RIP), which states that a given operator does not distort a matrix too severely when projected on to the tangent space of the ground truth.

\begin{remark}
    More mathematically, let $\mathcal{K}_\Omega$ be a stochastic sampling operator, such as $\Po$, $\Ro$, or $\Fo$. RIP states that with high probability
\begin{equation*}
    \left\Vert\Pt\mathcal{K}_\Omega\Pt-c\Pt\right\Vert\leq\varepsilon_0,
\end{equation*}
for some $\varepsilon_0>0$ and some constant $c>0$.  In practice, this statement is proven using non-commutative concentration inequalities, first introduced in the matrix completion literature in \cite{recht2011simpler,gross2010note}, requiring that $\mathcal{K}_\Omega$ is a sum of i.i.d. random operators and that the expectation of $\Pt\mathcal{K}_\Omega\Pt = c\Pt$ for some $c>0$. It is well-established that $\Po$ possesses this property, and in this paper we show that $\Ro$ also exhibits RIP. The proof, seen in Theorem~\ref{thm: R_omega RIP}, is completed using standard techniques with some specific properties of the dual bases $\{\wa\}_{\ai}$ and $\{\va\}_{\ai}$, and RIP for $\Ro$ is established with only slightly worse sample complexity than for $\Po$. Proving a similar statement for $\Fo$ is challenging, as $\mathbb{E}[\Fo] \neq c\mathcal{I}$. However, numerical evidence strongly indicates that $\Vert\Pt\Fo\Pt - \frac{m}{L}\Pt\Vert$ is in fact small for random ground truth matrices in $\mathbb{S}$,  and the subsequent convergence analysis of Algorithm~\ref{alg:F_omega descent} will be the subject of future work.
\end{remark}

To leverage the analytical properties of $\Ro$ while sidestepping the technical challenges of its non-self-adjoint nature, we define an algorithm by analogy to Algorithm~\ref{alg:F_omega descent} but without any reference to an objective function. As we will show in Section~\ref{sec: Analysis results}, this algorithm will give us strong convergence guarantees with reasonable sample complexities at a cost of interpretability. As such, we define Algorithm~\ref{alg:R_omega descent} to reconstruct a ground truth matrix $\X$ as follows:

\begin{algorithm}[ht!]
\caption{Riemannian Pseudo-Gradient Descent}\label{alg:R_omega descent}
\begin{algorithmic}
\STATE\textbf{Initialization}: $\X_0 = \U_0\D_0\U_0^\top$
\FOR {$l = 0,1,...$}
    \STATE 1. $\G_l = \Ro(\X-\X_l)$
    \STATE 2. $\alpha_l = \mathrm{max}\left\{\frac{\Vert \Ptl\G_l\Vert_\fro^2}{\la \Ptl\G_l,\Ro\Ptl\G_l\ra},0\right\}$
    \STATE 3. $\W_l = \X_l + \alpha_l \Ptl\G_l$
    \STATE 4. $\X_{l+1} = \mathcal{H}_r(\W_l)$
\ENDFOR
\STATE\textbf{Output:} $\X_\mathrm{rev}$
\end{algorithmic}
\end{algorithm}

Unlike in Algorithm~\ref{alg:F_omega descent}, we cannot compute the steepest descent of an objective function in $\T_l$, so we consider a surrogate modeled after each of the preceding algorithms. The maximum with zero in Step 2 of the algorithm is introduced to avoid divergence, as $\Ro$ is not a positive semi-definite operator and the denominator cannot be guaranteed to be positive for arbitrary points in $\mfr$. When $\alpha_l = 0$ occurs, the algorithm terminates. Positive $\alpha_l$ is required for convergence, and the conditions are provided and characterized in Lemma~\ref{lem:Stepsize}. This condition is satisfied in the high-sample regime, where $\varepsilon_0$ is small.

In both of the preceding approaches, the thin spectral decomposition in the gradient descent scheme is the most expensive, especially when $n$ is large. As described previously, the authors in \cite{wei2020guarantees} found an efficient way to reduce the computational complexity of this decomposition from $\mathcal{O}(rn^2)$ to $\mathcal{O}(r^3)+\mathcal{O}(nr^2)$, substantially reducing the cost per iteration, which we implement as well.

Computation of $\Ro(\X)$ can be done efficiently, with a minimal complexity per iteration. This is because a given iterate $\X_l$ can be easily translated to its distance matrix $\D_l$ via \eqref{eqn: X to D}, and through \eqref{eqn: Ro and Po equiv}, $\Ro(\X)$ can be computed in $\mathcal{O}(m)$ operations, for $\vert\Omega\vert = m$. From \cite{Smith2023}, it can be shown that the total cost per loop is approximately $\mathcal{O}(m) + \mathcal{O}(n^2) + \mathcal{O}(mr)+\mathcal{O}(nr^2)+\mathcal{O}(r^3)$ for an $n\times n$ rank-$r$ matrix.

\section{Theoretical Analysis}\label{sec: Analysis results}
In this section, we will provide the main results of this work, which are the local convergence and recovery guarantees for Algorithm~\ref{alg:R_omega descent}, presented in Theorems~\ref{thm: Local Convergence},~\ref{thm: Recovery Guarantee I}, and~\ref{thm: Resampling Recovery Guarantee}. Prior to these guarantees, we will first introduce slightly altered incoherence conditions for the non-orthogonal problem at hand. Pathological cases can arise where a ground truth matrix $\X$ has few non-zero coefficients in a dual basis expansion, which can cause issues in the recovery of said matrix from samples. This is well-studied in the standard matrix completion problem, and is captured in the idea of incoherence with respect to the standard basis. Relating incoherence to the underlying geometry of points is an interesting problem, but this is outside the scope of the current work.
\begin{asp}[Incoherence assumption]\label{asp: Incoherence assumption}
    Let $\X\in\Rnn$ be a rank-$r$ matrix with eigenvalue decomposition $\X = \U\D\U^\top$. We assume that $\X$ is $\frac{\nu}{4}$-incoherent to the basis $\{\wa\}_{\alphab\in\mathbb{I}}$, $\nu$-incoherent to its dual basis $\{\va\}_{\alphab\in\mathbb{I}}$, and $\frac{\nu}{64}$-incoherent in the standard matrix basis; that is, there exists a constant $\nu\geq 1$ such that for all $\alphab = (i,j)\in\mathbb{I}$:
    \begin{equation}\label{eq: Incoherence equations}
        \left\Vert \Pu \e_{ij}\right\Vert_\fro\leq \sqrt{\frac{\nu r}{128n}}, \qquad \left\Vert\Pu\wa\right\Vert_\fro\leq\sqrt{\frac{\nu r}{8n}},\quad \mathrm{and} \quad \left\Vert\Pu\va\right\Vert_\fro\leq\sqrt{\frac{\nu r}{2n}}.
    \end{equation}
    In addition to the above, we require that
        \begin{equation}\label{eq: Pt incoherence equations}
        \left\Vert \Pt \e_{ij}\right\Vert_\fro\leq \sqrt{\frac{\nu r}{128n}}, \qquad \left\Vert\Pt\wa\right\Vert_\fro\leq\sqrt{\frac{\nu r}{8n}},\quad \mathrm{and} \quad\left\Vert\Pt\va\right\Vert_\fro\leq\sqrt{\frac{\nu r}{2n}}.
    \end{equation}
\end{asp}
This assumption is in accordance with both the standard definitions of incoherence as $\left\Vert \Pu\e_{ij}\right\Vert_\fro\leq\left\Vert \Pu\e_i\right\Vert_2$. Additionally, notice that these two definitions are equivalent up to a small constant, as 
\begin{align*}
    \left\Vert\Pt\wa\right\Vert_\fro = \left\Vert \Pu\wa + \wa\Pu - \Pu\wa\Pu\right\Vert_\fro\leq3\left\Vert\Pu\wa\right\Vert_\fro,
\end{align*}
where the first inequality follows from the triangle inequality and the self-adjointness of $\Pu\wa$. As such, we pick a $\nu$ large enough such that the inequalities in \eqref{eq: Incoherence equations} and \eqref{eq: Pt incoherence equations} hold. As in \cite{wei2020guarantees}, we note that the first condition above implies the following:
\begin{align*}
    \left\Vert \Pu \e_{ij}\right\Vert_\fro^2 = \la \Pu\e_{ij},\Pu \e_{ij}\ra
    =\la \Pu\e_{ij},\e_{ij}\ra
    =\Tr\left(\e_{ij}\e_{ji}\Pu\right)
    =\la\Pu,\e_{ii}\ra
    = \left\Vert \U^{(i)}\right\Vert_2^2.
\end{align*}
This indicates that $\left\Vert\U^{(i)}\right\Vert^2_2\leq\frac{\nu r}{128 n}$, which will be relevant when discussing the initialization of Algorithm~\ref{alg:R_omega descent} using a trimming step in Algorithm~\ref{alg:Trimming Algorithm}. 

\begin{remark}\label{rem: incoherence remark}
    We want to note that the first assumption in \eqref{eq: Incoherence equations} actually implies the next two. That is to say, if 
    $\Vert \Pu\e_{ij}\Vert_2\leq\sqrt{\frac{\nu r}{128n}}$, then by the triangle inequality
    \[
    \Vert \Pu\wa\Vert_\fro \leq 4\max_{(i,j)\in\mathbb{I}}\Vert\Pu\e_{ij}\Vert_\fro\leq \sqrt{\frac{\nu r}{8n}}.
    \]
    To see the last result, notice that
        \begin{align*}
        \Vert\Pu\va\Vert_\fro &= \left\Vert \Pu\left(\sum_{\betab\in\mathbb{I}}H^{\alphab\betab}\wb\right)\right\Vert_\fro\\
        &\leq\sum_{\betab\in\mathbb{I}}\left\vert H^{\alphab\betab}\right\vert\Vert\Pu\wb\Vert_\fro\\
        &\leq \sqrt{\frac{\nu r}{8n}}\sum_{\betab\in\mathbb{I}}\left\vert H^{\alphab \betab}\right\vert,
    \end{align*}
    and as $\sum_{\ai}\vert H^{\alphab\betab}\vert \leq 2$ from Lemma~\ref{lem: H and H^-1 eigvals}, the claim follows. A similar proof shows the same relationship for the equations in \eqref{eq: Pt incoherence equations}.
\end{remark}

Additionally, we make an assumption in accordance with \cite{wei2020guarantees}:
\begin{asp}\label{asp: mu_1 assumption}
    Let $\X\in\Rnn$ be a rank-$r$ matrix. We assume that an absolute numerical constant $\mu_1$ such that
    \begin{equation} \label{eqn: mu_1 definition}
        \left\Vert\X\right\Vert_\infty\leq\mu_1\sqrt{\frac{r}{n^2}}\left\Vert \X\right\Vert.
    \end{equation}
\end{asp}

\begin{remark}
    Notice that this condition is equivalent up to scaling factors for a similar assumption in \cite{keshavan2010matrix}, which itself can be upper bounded by a similar coherence condition in \cite{candes2009exact}. In fact, we can relate $\mu_1$ directly to $\nu$ as follows. For simplicity and relevance to our problem, let $\X \succeq \bm{0}$, and notice that
\begin{align*}
    \displaystyle\frac{\Vert \X\Vert_\infty}{\Vert\X\Vert} &= \displaystyle\frac{1}{\Vert\X\Vert} \max_{ij} \vert X_{ij}\vert\\
    &=\displaystyle\frac{1}{\Vert\X\Vert}\max_{ij} \left\vert\sum_{kl} U_{ik}D_{kl}U_{jl}\right\vert\\
    &\leq\max_{ij} \sum_{k=1}^r \left\vert U_{ik}\frac{\lambda_k}{\lambda_1}U_{jk}\right\vert\\
    &\leq \max_{ij} \sum_{k=1}^r \left\vert U_{ik}U_{jk}\right\vert\\ 
    &\leq \sqrt{\sum_{1\leq k\leq r}\vert U_{ik}\vert^2}\sqrt{\sum_{1\leq k\leq r}\vert U_{ik}\vert^2}\\
    &\leq \frac{\nu r}{128 n},
\end{align*}
where the penultimate inequality follows from Cauchy-Schwartz, and the final inequality from \eqref{eq: Incoherence equations}, indicating that in a worst-case scenario $\mu_1\leq\frac{\nu \sqrt{r}}{128}$. This property is ultimately separate from the definition of $\nu$, but at the very least it can be upper bounded as a function of $\nu$.

\end{remark}

Additionally, we are typically interested in large $n$. Assuming that $n\geq 3$ produces uniform results for several bounds in the appendix, and is formally stated as an assumption.
\begin{asp}\label{asp: size}
    For the given ground truth rank-$r$ matrix $\X\in\Rnn$, we assume that $n\geq 3$.
\end{asp}
Throughout the remainder of this work, we will assume that our ground truth matrix $\X\in\mathbb{S}$ satisfies both Assumptions~\ref{asp: Incoherence assumption} and~\ref{asp: mu_1 assumption} with $\cO(1)$ constant factors $\nu$ and $\mu_1$. As in \cite{wei2020guarantees}, we identify a neighborhood in $\mathcal{M}_r$ around which any initial guess in this neighborhood converges linearly to the true solution with high probability.

As mentioned previously, the most critical property for convergence of Algorithm~\ref{alg:R_omega descent} is RIP. This theorem provides the conditions needed for RIP of $\Ro$:
\begin{thm}[Restricted Isometry Property (RIP) for $\Pt\Ro\Pt$]\label{thm: R_omega RIP}
    With probability at least $1-2n^{1-\beta}$,
    \begin{equation*}
        \frac{L}{m}\left\Vert \Pt\Ro\Pt - \frac{m}{L}\Pt\right\Vert \leq \sqrt{\frac{8\beta\nu^2 r^2 n\log(n)}{3m}},
    \end{equation*}
    for $m\geq \frac{8}{3}\beta\nu^2 r^2 n\log(n)$. In particular, 
    \begin{equation*}
        \frac{L}{m}\left\Vert \Pt\Ro\Pt - \frac{m}{L}\Pt\right\Vert \leq \varepsilon_0,
    \end{equation*}
    for any $\varepsilon_0>0$ if $m\geq \frac{8}{3}\beta\left(\frac{\nu r}{\varepsilon_0}\right)^2 n\log(n)$.
    Additionally, under the same conditions as above, we also have
        \begin{equation*}
        \frac{L}{m}\left\Vert \Pt\Ro^\star\Pt - \frac{m}{L}\Pt\right\Vert \leq \varepsilon_0.
    \end{equation*}
\end{thm}
\begin{proof}[Proof sketch]
   This result is primarily a consequence of Theorem~\ref{thm:Bernstein}, the non-commutative Bernstein inequality. As $\mathbb{E}(\Ro) = \frac{m}{L}\mathcal{I}$, we see that $\mathbb{E}(\Pt\Ro\Pt) = \frac{m}{L}\mathcal{\Pt}$, and the rest of the proof is leveraging specific properties of the dual bases $\{\wa\}$ and $\{\va\}$ to prove concentration of $\Pt\Ro\Pt$ around its expectation. See Appendix~\ref{proof: R_omega RIP} for details.
\end{proof}

This leads us into the following theorem, which is the crux of this work. Theorem~\ref{thm: Local Convergence}, stated with a brief proof outline in the main text, describes a local attractive basin around the ground truth solution, provided that $\Pt\Ro\Pt$ exhibits RIP. With high probability, the algorithm will converge to the ground truth from any initialization in this attractive basin. The full proof is delayed to Appendix~\ref{proof: Local convergence}. 

\begin{thm}[Local Convergence of Algorithm~\ref{alg:R_omega descent}]\label{thm: Local Convergence}
    Let $\X\in\Rnn$ be the measured rank-$r$ matrix and let $\T$ be the tangent space of $\mathcal{M}_r$ at $\X$. Suppose that
    \begin{gather}
            \left\Vert\frac{L}{m}\Pt\Ro\Pt-\Pt\right\Vert\leq\varepsilon_0 \label{eqn: RIP assumption}\\
    \frac{\Vert\X_l-\X\Vert_\fro}{\sigma_\text{min}(\X)}\leq\frac{\sqrt{m}\varepsilon_0}{16n^{5/4}\sqrt{\beta\nu r\log{n}}}\label{eqn: neighborhood assumption}\\
    \Vert\Ro\Vert\leq \frac{m}{L} + 4\sqrt{\frac{8m\log(n)}{n}}\label{eqn: R_o assumption}\\
    \Vert \Ro\Pt\Vert \leq \frac{m}{L} + \frac{m\sqrt{n}}{L}\sqrt{\frac{\beta\nu r n\log(n)}{3m}}\label{eqn: RoPt assumption}\\
    \Vert \Pt\Ro\Vert\leq \frac{m}{L}+\frac{4m\sqrt{n}}{L}\sqrt{\frac{\beta \nu r n\log(n)}{3m}},\label{eqn: PtRo assumption}
    \end{gather}
        where $\varepsilon_0$ is a constant satisfying
        \[
        \delta=\frac{18\varepsilon_0}{1-4\varepsilon_0}<1.
        \]
        Then the algorithm converges linearly as the iterates satisfy
        \[
        \Vert\X_{l+1} - \X\Vert_\fro \leq \delta^l\Vert\X_0-\X\Vert_\fro.
        \]
\end{thm}
\begin{proof}[Proof sketch of Theorem~\ref{thm: Local Convergence}]
    The theorem begins first by simple linear algebra, as we have
    \begin{align*}
        \Vert \X_{l+1} -\X\Vert_\fro &= \Vert\X_{l+1} -\W_l -\X + \W_l\Vert_\fro\\
        &\leq \Vert \X_{l+1} - \W_l\Vert_\fro + \Vert \X-\W_l\Vert_\fro\\
        &\leq 2\Vert \W_l-\X\Vert_\fro,
    \end{align*}
    where the last inequality follows from $\X_{l+1}$ being the best rank-$r$ approximation to $\W_l$ by Eckart-Young-Mirsky \cite{eckart1936approximation}. Next, plugging in $\W_l = \X_l + \alpha_l\Ptl\G_l$, we see that
    \begin{align*}
        \Vert \X_{l+1}-\X\Vert_\fro&\leq 2\left\Vert \X_l + \alpha_l\Ptl\G_l - \X\right\Vert_\fro\\
        &=2\Vert \X_l-\X -\alpha_l\Ptl\Ro(\X_l-\X)\Vert_\fro\\
        &\leq \underbrace{2\Vert (\Ptl - \alpha_l\Ptl\Ro\Ptl)(\X_l-\X)\Vert_\fro}_{I_1}\\
        &\quad+ \underbrace{2\Vert(I-\Ptl)(\X_l-\X)\Vert_\fro}_{I_2}\\
        &\quad+ \underbrace{2\vert\alpha_l\vert\Vert\Ptl\Ro(I-\Ptl)(\X_l-\X)\Vert}_{I_3}.
    \end{align*}
The remainder of the proof is in the bounding of $I_1$, $I_2$, and $I_3$. $I_1$ is proven by showing that in a neighborhood of the solution, defined by \eqref{eqn: neighborhood assumption}, a local form of RIP for $\Ro$ holds if \eqref{eqn: RIP assumption} is true. This proof leverages the assumptions made in \eqref{eqn: R_o assumption}, \eqref{eqn: RoPt assumption}, and \eqref{eqn: PtRo assumption}. $I_2$ follows from the neighborhood assumption of \eqref{eqn: neighborhood assumption} in tandem with Lemma~\ref{lem: Projection Bounds}, and $I_3$ follows from bounds on the step size (seen in Lemma~\ref{lem:Stepsize}), the assumption in \eqref{eqn: PtRo assumption}, and Lemma~\ref{lem: Projection Bounds}. The assumptions in \eqref{eqn: RIP assumption}, \eqref{eqn: R_o assumption}, \eqref{eqn: RoPt assumption}, and \eqref{eqn: PtRo assumption} are all proven via high probability guarantees using Theorem~\ref{thm:Bernstein}. The technical details are deferred to the appendix, but Figure~\ref{fig: Convergence proof schematic} highlights the main dependencies of each lemma and how they work into the overall convergence. See the full proof in~\ref{proof: Local convergence}.
\end{proof}

\begin{figure}[ht!]\label{fig: Convergence proof schematic}
\centering
\begin{tikzpicture}[
    node distance=1cm,
    every node/.style={rectangle, draw, minimum width=.8cm, minimum height=.8cm, text centered, font =\small},
    arrow/.style={thick,->,>=stealth}
    ]

    \node (Ber)[text width=3.5cm] {Theorem~\ref{thm:Bernstein} \\Non-commutative Bernstein Inequality};
    \node (waprop)[above of=Ber,yshift=2.5cm,text width = 2.5cm]{Section~\ref{appendix: dual basis} Properties of the dual basis };
    \node (RIP)[right=of Ber,yshift = 2cm,xshift = -.5cm,text width = 2.5cm]{Theorem~\ref{thm: R_omega RIP} $\Ro$ RIP};
    \node (LRIP)[right of=RIP,xshift=3cm,text width=2.6cm]{Lemma~\ref{lem: New Local Pseudo proof} Local RIP of $\Ro$};
    \node (step)[below of =LRIP,yshift=-.3cm,text width = 2.5cm]{Lemma~\ref{lem:Stepsize} Stepsize bound};
    \node (I3)[right of =LRIP,xshift=2.5cm]{$I_3$ bound};
    \node (I1)[below of=I3,yshift = -.3cm]{$I_1$ Bound};
    \node (I2)[below of =I1,yshift=-.3cm]{$I_2$ bound};
    \node (proj)[below of =step,yshift=-.3cm,text width = 2.8cm]{Lemma~\ref{lem: Projection Bounds} Projection bounds \cite{wei2020guarantees}};
    \node (1SHT)[below right=of Ber,text width = 2.8cm,xshift=-.4cm]{Lemma~\ref{lem:Initialization}\\ One-step hard threshold bound};
    \node (convergence)[right of=I3, xshift = 2cm,text width = 3cm]{Theorem~\ref{thm: Local Convergence}\\ Local Convergence of Algorithm~\ref{alg:R_omega descent}};
    \node (asymRIP)[above left=of LRIP,xshift = 1cm,text width=2.8cm]{Lemma~\ref{lem: Asymmetric RIP} Asymmetric RIP of $\Ro$};
    \node(trim)[right of=asymRIP,xshift = 4.7cm,text width = 2.5cm]{Lemma~\ref{lem: Trimming result} Trimming bound};
    \node (trimsc)[above of=convergence,yshift = .75cm, text width = 2.5cm]{Theorem~\ref{thm: Resampling Recovery Guarantee} Resampling initialization guarantees};
    \node (1SHT cpxty)[below of=convergence,text width = 2.5cm,yshift = -2.95cm]{Theorem~\ref{thm: Recovery Guarantee I} One-step hard thresholding initialization guarantees};

    \draw [arrow] (Ber) -- (RIP);
    \draw [arrow] (waprop) -- (RIP);
    \draw [arrow] (RIP) -- (LRIP);
    \draw [arrow] (LRIP) -- (step);
    \draw [arrow] (step) -- (I1);
    \draw [arrow] (proj) -- (I2);
    \draw [arrow] (LRIP) -- (I3);

    \draw [arrow] (I1) -- (convergence);
    \draw [arrow] (I2) -- (convergence);
    \draw [arrow] (I3) -- (convergence);
    \draw [arrow] (1SHT) -- (1SHT cpxty);
    \draw [arrow] (convergence) -- (1SHT cpxty);
    \draw [arrow] (RIP) -- (asymRIP);
    \draw [arrow] (trim) -- (trimsc);
    \draw [arrow] (asymRIP) -- (trim);
    \draw [arrow] (convergence) -- (trimsc);
    \draw [arrow] (Ber) -- (1SHT);
\end{tikzpicture}

\caption{This diagram is a schematic of the overall proof of convergence. Arrows indicate how results depend on one another, and how they link together to form the overall proof of convergence. Not every exact dependency is shown in this figure for legibility purposes, instead focusing on the key pieces of the overall flow of the argument.}

\end{figure}
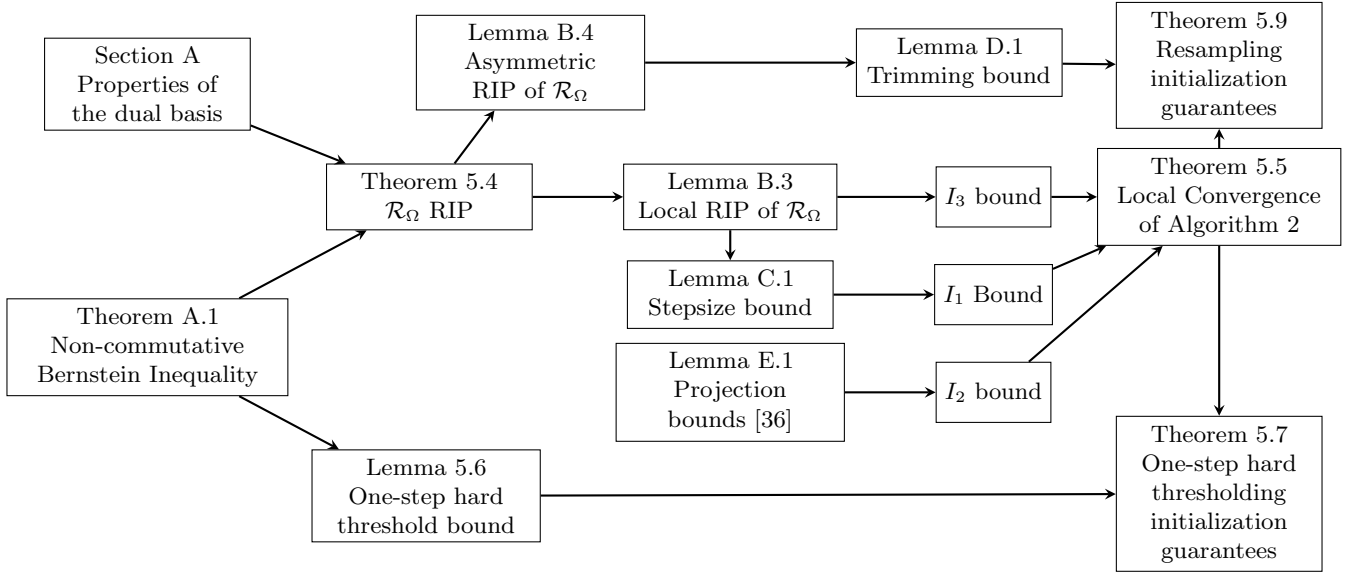

In Theorem~\ref{thm: Local Convergence}, the linear convergence rate required that $\varepsilon_0<\frac{1}{22}$, which is satisfied for $m> 1300\nu^2 r^2 n\log(n)$, a smaller constant factor than that required of the following initialization guarantees. That is to say, the local neighborhood guarantees are stricter in practice than the RIP requirement.

\subsection{Initialization}
Given that the convergence of this algorithm is only local, initialization is important to consider in the context of sample complexity. The simplest initialization, a hard thresholding to $\mathcal{M}_r$ of the measured information, provides a reasonable starting point. The following theorem describes how close such an initialization might be to the ground truth.

\begin{lem}[Initialization via One Step Hard Thresholding]\label{lem:Initialization}
Let $\X\in\Rnn$ be the underlying measured rank-$r$ matrix, and let $\X_0 = \frac{L}{m}\mathcal{H}_r(\Ro(X))$ with $\Omega = m \geq \frac{40}{3}\beta n\log{n}$. It follows that with probability at least $1-2n^{1-\beta}$ that
\begin{equation}\label{eqn:Initialization}
    \left\Vert \X_0-\X\right\Vert_\fro\leq \sqrt{\frac{320 r^2 \mu_1^2 n\log(n)}{3m}}\left\Vert \X\right\Vert.
\end{equation}
\end{lem}
\begin{proof}
    See Section~\ref{proof: 1SHT}.
\end{proof}

This result shows that with probability at least $1-2n^{1-\beta}$, the assumption of Lemma~\ref{lem: New Local Pseudo proof} is satisfied if 
\begin{equation*}
    m\geq 170\frac{\kappa \mu_1\sqrt{\beta\nu r^3} }{\varepsilon_0} n^{7/4}\log(n),
\end{equation*}
 where $\kappa = \frac{\left\Vert \X\right\Vert}{\sigma_{\mathrm{min}}(\X)}$ is the condition number of $\X$. 

This leads into the following theorem, which is one of the primary results of the work.
\begin{thm}[Recovery Guarantee I]\label{thm: Recovery Guarantee I}
    Suppose $\vert\Omega\vert = m$ with the indices sampled uniformly with replacement. Given an initialization of  $\X_0 = \frac{L}{m}\mathcal{H}_r(\Ro(\X))$ for the rank-$r$, $\nu$-incoherent ground truth matrix $\X$ with condition number $\kappa$, and given  
    \begin{equation*}
       m\geq\max\left\{\frac{8}{3}\frac{\sqrt{\beta\nu^3 r}}{\varepsilon_0},170\kappa \mu_1n^{3/4}\right\}\frac{\sqrt{\beta\nu r^3}}{\varepsilon_0}n\log(n),
    \end{equation*}
with $\beta>1$, then Algorithm~\ref{alg:R_omega descent} linearly converges to the ground truth $\X$ with probability at least $1-10n^{1-\beta}$.
\end{thm}
\begin{proof}
    This follows from taking the local neighborhood assumption seen in Theorem~\ref{thm: Local Convergence}. By setting the bound produced in Lemma~\ref{lem:Initialization} to be less than required for local convergence, the result follows after some minor algebra and taking the maximum with the sample complexity requirements seen in Theorem~\ref{thm: R_omega RIP}.
\end{proof}
This naive one-step hard threshold initialization can be improved again following a construction in \cite{wei2020guarantees}, using a resampling and trimming algorithm, both defined as follows:

\begin{algorithm}[ht!]
\caption{Riemannian Resampling for Initialization}\label{alg:Resampling Algorithm}
\begin{algorithmic}
\STATE\textbf{Partition} $\Omega$ into $S+1$ equal groups $\Omega_0,\Omega_1,...,\Omega_S$, each of size $\hat{m}$\\{\bf Set} $\Z_0 = \mathcal{H}_r\left(\frac{L}{\hat{m}}\mathcal{R}_{\Omega_0}(\X)\right)$
\FOR {$l = 0,1,...,S-1$}
    \STATE 1. $\hat{\Z}_l = \texttt{trim}(\Z_l)$
    \STATE 2. $\Z_{l+1} = \mathcal{H}_r\left(\hat{\Z}_l + \frac{L}{\hat{m}}\mathcal{P}_{\hat{\mathbb{T}}_l} \mathcal{R}_{\Omega_{l+1}}(\X-\hat{\Z}_l)\right) $
\ENDFOR
\STATE\textbf{Output:} $\X_0 = \Z_S$
\end{algorithmic}
\end{algorithm}

\begin{algorithm}[ht!]
\caption{\tt{trim}}\label{alg:Trimming Algorithm}
\begin{algorithmic}[]
\STATE{\bf Input:} $\Z_l = \U_l\D_l\U_l^\top$
\STATE{\bf Output:} $\widehat{\Z}_l=\A_l\D_l \A_l^\top$, where 
$\A_l^{(i)} = \frac{\U_l^{(i)}}{\left\Vert \U_l^{(i)}\right\Vert_2}\min\left\{ \left\Vert \U_l^{(i)}\right\Vert_2,\sqrt{\frac{\nu r}{n}}\right\}
$
\end{algorithmic}
\end{algorithm}

This trimming algorithm is a projection onto the space of matrices that are 
$\nu$-incoherent with respect to the standard matrix basis, not necessarily with respect to the basis $\{\wa\}_{\ai}$. However as noted previously, the incoherence parameter differs by at most an $\mathcal{O}(1)$ constant, so this is a reasonable surrogate, especially for large $n$.

We can analyze Algorithm~\ref{alg:Resampling Algorithm} and get the following result: 

\begin{lem}[Riemannian Resampling Result]\label{lem:Resampling result}
Let $\X\in\Rnn$ be the measured rank-$r$ matrix with condition number $\kappa$. Let $S$ be the number of partitions specified in Algorithm~\ref{alg:Resampling Algorithm}, and let $\hat{m} = \frac{m}{S+1}$. Then for all $\beta>1$, with probability at least $1-(2+4S)n^{1-\beta}$ the output of Algorithm~\ref{alg:Resampling Algorithm} satisfies
\begin{equation*}
    \left\Vert \X_0 - \X\right\Vert_\fro\leq \left(\frac{5}{6}\right)^S\frac{\sigma_\mathrm{min}(\X)}{256\kappa^2},
\end{equation*}
provided that $\hat{m}\geq \max\left\{ (1.61\times10^5)\nu^2, (7.77\times10^5)\kappa^4\mu_1^2 \right\}\kappa^2 r^2 n\log(n)$.
\end{lem}
\begin{proof}
    See \ref{proof: Resampling proof}.
\end{proof}

Assuming that $m\geq \beta\nu r n\log(n)$, which is relaxed from the requirement for RIP seen in Theorem~\ref{thm: R_omega RIP}, 
Lemma~\ref{lem:Resampling result} shows that taking
\begin{equation*}
    S\geq 6\log\left(\frac{n^{3/4}}{16\varepsilon_0}\right),
\end{equation*}
the third condition in Lemma~\ref{lem: New Local Pseudo proof} can be satisfied with probability at least $1-\left(6+24\log\left(\frac{n^{3/4}}{16\varepsilon_0}\right)\right)n^{1-\beta}$ for a large enough sample complexity. This leads to the final recovery guarantee, attenuating the $n$ dependence in the sample complexity.
\begin{thm}\label{thm: Resampling Recovery Guarantee}
    Let $\X\in\Rnn$ be the measured rank-$r$, $\nu$-incoherent matrix with condition number $\kappa$, and suppose that $\vert\Omega\vert = m$ is a set of sampled indices from $\univ$ uniformly at random with replacement. Let $\X_0$ be the output of Algorithm~\ref{alg:Resampling Algorithm}. Then for any $\beta>1$, the iterates of Algorithm~\ref{alg:R_omega descent} linearly converge to $\X$ with probability at least $1-\left(6+24\log\left(\frac{n^{3/4}}{16\varepsilon_0}\right)\right)n^{1-\beta}$  provided that
    \begin{equation*}
        m \geq \max\left\{\frac{8\nu^2}{3\varepsilon_0^2},(2.3\times10^5)\nu^2\log\left(\frac{n^{3/4}}{16\varepsilon_0}\right),(1.2\times10^6)\kappa^4\mu_1^2\log\left(\frac{n^{3/4}}{16\varepsilon_0}\right)  \right\}\kappa^2 r^2 n\log(n).
    \end{equation*}
\end{thm}
\begin{proof}
    As in Theorem~\ref{thm: Recovery Guarantee I}, this proof follows from the local neighborhood conditions of
    Theorem~\ref{thm: Local Convergence}, combined with the sample complexity results from Lemma~\ref{lem:Resampling result} and Theorem~\ref{thm: R_omega RIP}. The constants are not optimized, and could be further improved.
\end{proof}

\section{Numerical Results} \label{sec: Numerics}
In this section, we test the proposed algorithms on synthetic and real data.
\subsection{Synthetic Data Experiments}
To test Algorithms~\ref{alg:F_omega descent} and \ref{alg:R_omega descent}, various two and three dimensional datasets were used, and are referred to in Table~\ref{tab: synthetic IPM means table} with their corresponding sizes. The goal of Algorithms~\ref{alg:F_omega descent} and~\ref{alg:R_omega descent} is to recover the full set of points $\Pb$ up to orthogonal transformation by sampling the entries above the diagonal of $\D$ uniformly with replacement, with a total of $\gamma L$ entries chosen for $\gamma\in[0,1]$. Algorithm~\ref{alg:R_omega descent} reconstructs the Gram matrix $\X = \P^\top\P$, from which $\P$ can be recovered. The comparison referenced in Table~\ref{tab: synthetic IPM means table} is the relative error between the recovered matrix $\X_\mathrm{rev}$ and the ground truth matrix $\X$ in Frobenius norm. Each run was terminated at either 1000 iterations or when a relative Frobenius norm difference between iterates of $10^{-5}$ was achieved.

\begin{table}[ht!]
\vspace{-0.15in}
\centering
\caption{Relative recovery error $\left\Vert \X-\X_\mathrm{rev}\right\Vert_\fro/\left\Vert \X\right\Vert_\fro$ between the recovered Gram matrix and the true Gram matrix averaged over 25 trials using Algorithms~\ref{alg:F_omega descent}, \ref{alg:R_omega descent}, and the non-convex algorithm in \cite{tasissa2018exact}. \vspace{0.05in}} \label{tab: synthetic IPM means table}
\begin{tabular}{||l|c|c|c |c|c| c||} 
\hline
\diagbox[width=10em]{Dataset}{$\gamma$}&
   {10\%}&  {7\%}& {5\%} & {3\%} & {2\%} & {1\%}  \\
\hline
\hline
\multicolumn{7}{||c||}{Algorithm~\ref{alg:F_omega descent}}\\
\hline
 Sphere (3D, $n=1002$)& 2.92e-05 &   3.69e-05 &   6.01e-05 &  1.28e-04 &   6.82e-03  &  9.11e-01 \\
\hline
Cow (3D, $n=2601$)& 3.38e-05&  4.06e-05   & 4.89e-05   & 7.61e-05 &   1.07e-04  &  8.60e-03 \\
\hline
Swiss Roll (3D, $n=2048$)& 2.92e-05& 3.97e-05& 5.06-05& 7.71e-05 & 1.21e-04&5.52e-02\\
\hline
 U.S. Cities (2D, $n=2920$)&  4.10e-05  &  4.67e-05  &  2.01e-03 & 6.94e-03 &   1.63e-02  &  6.08e-02 \\
\hline
\hline
\multicolumn{7}{||c||}{Algorithm~\ref{alg:R_omega descent}}\\
\hline
 Sphere (3D, $n=1002$)& 1.53e-05  &  3.86e-05  &  2.11e-04  &  7.88e-02  &  2.29e-01   & 1.78e+00\\
\hline
Cow (3D, $n=2601$)& 4.15e-02  &  3.40e-02 &   6.28e-02  &  1.76e-01   & 3.60e+00 &  7.73e-01 \\
\hline
Swiss Roll (3D, $n=2048$)& 8.34e-06 &  1.46e-05  &  3.00e-05 &  1.52e-03 &  2.82e-01 &  9.73e-01\\
\hline
 U.S. Cities (2D, $n=2920$)& 3.02e-02  &  1.23e-01  &  1.46e-01  &  1.86e-01&   3.13e-01 &  8.35e-01 \\
\hline
\hline
\multicolumn{7}{||c||}{Non-convex algorithm in \cite{tasissa2018exact}}\\
\hline
 Sphere (3D, $n=1002$)& 6.14e-06 &   9.86e-06 &   1.36e-05&  3.04e-05&  6.18e-05  &  1.00e-01 \\
\hline
Cow (3D, $n=2601$)& 5.73e-06&  7.66e-06   & 1.06e-05   & 1.65e-05 &   2.11e-05 &  4.46e-05 \\
\hline
Swiss Roll (3D, $n=2048$)& 2.19e-06& 1.22e-06&1.01e-06& 1.87e-06 & 1.06e-06& 3.34e-05\\
\hline
 U.S. Cities (2D, $n=2920$)&  4.09e-07  & 6.09e-07 &  8.19e-07  &  1.32e-06 &   2.30e-06  &  4.69e-06 \\
 \hline
 \end{tabular}
\end{table}

In addition to the relative error comparison between the recovered Gram matrix and the ground truth Gram matrix, we compute the root mean square error (RMSE), defined as $\sqrt{\frac{1}{n}\Vert\P_\mathrm{rev}-\P\Vert_\fro^2}$ between the recovered point cloud $\P_\mathrm{rev}$ following a Procrustes realignment with the ground truth $\P$, under the same experimental parameters as with the Gram matrix recovery. The results are compiled in Table~\ref{tab: synthetic rmse table}.

\begin{table}[!ht]
\vspace{-0.15in}
\centering
\caption{RMSE between $\P_\mathrm{rev}$ and $\P$ averaged over 25 trials using Algorithms~\ref{alg:F_omega descent}, \ref{alg:R_omega descent}, and the non-convex algorithm in \cite{tasissa2018exact}. \vspace{0.05in}} \label{tab: synthetic rmse table}
\begin{tabular}{||l|c|c|c |c|c| c||} 
\hline
\diagbox[width=10em]{Dataset}{$\gamma$}&
   {10\%}&  {7\%}& {5\%} & {3\%} & {2\%} & {1\%}  \\
\hline
\hline
\multicolumn{7}{||c||}{Algorithm~\ref{alg:F_omega descent}}\\
\hline
 Sphere (3D, $n=1002$)& 2.06e-05 &   2.61e-05  &  4.24e-05  &  9.07e-05  &  4.82e-03&   7.14e-01\\
\hline
Cow (3D, $n=2601$)& 3.98e-05  & 4.89e-05  &  6.08e-05  &  9.25e-05  &  1.31e-04  &  1.24e-02 \\
\hline
Swiss Roll (3D, $n=2048$)& 2.46e-04  &  3.36e-04  &  4.26e-04 &  6.48e-04  &  1.13e-03  &  4.92e-01\\
\hline
 U.S. Cities (2D, $n=2920$)&7.12e-04  &  8.13e-04  &  5.39e-02  &  1.86e-01 &   4.37e-01  &  1.70e+00 \\
 \hline
\hline
\multicolumn{7}{||c||}{Algorithm~\ref{alg:R_omega descent}}\\
\hline
 Sphere (3D, $n=1002$)& 1.08e-05  &  2.72e-05  & 1.49e-04  &  5.37e-02  &  1.42e-02 &   1.24e+00\\
\hline
Cow (3D, $n=2601$)& 6.67e-02  &  7.81e-02  &  1.38e-01  &  2.49e-01  &  5.87e-01  &  5.52e-01 \\
\hline
Swiss Roll (3D, $n=2048$)& 6.91e-05  &  1.21e-04  &  2.51e-04  &  1.24e-02  &  1.86e+00 & 1.34e+01\\
\hline
 U.S. Cities (2D, $n=2920$)& 9.94e-01  &  3.89e+00  &  5.10e+00  &  6.25e+00 &   7.77e+00 &  1.13e+01 \\
\hline
\hline
\multicolumn{7}{||c||}{Non-convex algorithm in \cite{tasissa2018exact}}\\
\hline
 Sphere (3D, $n=1002$)& 4.29e-06  &  6.94e-06  & 9.61e-06  &  2.14e-05  &  4.37e-05 &   7.10e-02\\
\hline
Cow (3D, $n=2601$)& 8.86e-06  &  1.19e-05  &  1.75e-05  &  2.70e-05  &  3.47e-05  &  7.08e-05 \\
\hline
Swiss Roll (3D, $n=2048$)& 2.02e-05  &  1.10e-05  &  8.90e-06  &  1.66e-05  &  9.32e-06 & 2.96e-04\\
\hline
 U.S. Cities (2D, $n=2920$)& 8.01e-06 &  1.40e-05  &  1.87e-05 & 3.11e-05&   5.18e-05 &  1.13e-04 \\
\hline
 \end{tabular}
\end{table}

As indicated by these experiments, Algorithm~\ref{alg:F_omega descent} more reliably reconstructs the underlying datasets from distance samples than Algorithm~\ref{alg:R_omega descent}, but both are outperformed by the non-convex algorithm in \cite{tasissa2018exact}. However, with the non-convex algorithm in \cite{tasissa2018exact} there is little hope of ever conducting local convergence analysis of this algorithm, whereas Algorithm~\ref{alg:R_omega descent} has been proven to exhibit local convergence. It remains to be seen if Algorithm~\ref{alg:F_omega descent} will be provably locally convergent, as determining a high-probability bound on $\left\Vert\Pt\Fo\Pt-c\Pt\right\Vert$ for some $c>0$ has proven challenging. As such, more practical utility lies in Algorithm~\ref{alg:F_omega descent} and \cite{tasissa2018exact} than in Algorithm~\ref{alg:R_omega descent}, but the strong theoretical results provided by Algorithm~\ref{alg:R_omega descent} will guide future work for convergence analysis of Algorithm~\ref{alg:F_omega descent} and other self-adjoint surrogates for $\Ro$.

\subsection{Experiments on Real Data}
Additional numerical experiments have been conducted on proteins, a common application of EDG, following the structured sampling method seen in \cite{lichtenberg_structured}. The sampling method on $n$ points has three different classes of points: $m$ pseudoanchors, 1 central anchor, and $n-m-1$ mobile nodes. The central anchor, corresponding to a row/column of the squared distance matrix, is fully known; that is, the distance between the central node and all $n$ points is revealed in the masked square distance matrix. The pairwise distances between the pseudoanchors are sampled with a Bernoulli probability $\gamma\in[0,1]$ for each entry, and each mobile node is connected uniformly at random to $k$ of the pseudoanchors. None of the distances between mobile nodes are known. The Gram matrix $\X$ is then reordered into the following pattern: the first $m$ rows/columns are the rows/columns corresponding to the pseudoanchors, the $m+1$-th row/column corresponds to the central anchor, and the remaining $n-m-1$ columns/rows correspond to the mobile nodes. This is illustrated in the figure below:
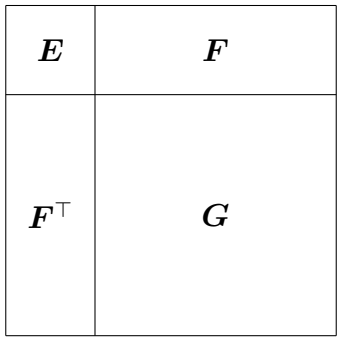
\begin{figure}[ht!]
    \centering
\begin{tikzpicture}[x=1pt, y=1pt, yscale=-1]
\draw (0,0) -- (125,0);
\draw (125,0) -- (125,125) ;
\draw (125,125) -- (0,125) ;
\draw (0,125) -- (0,0) ;
\draw (33.75,125) -- (33.75,0) ;
\draw (0,33.75) -- (125,33.75) ;
 \filldraw (16.875,16.875) circle (0pt)
                         node[font=\large]{$\E$};
 \filldraw (16.875,79.375) circle (0pt)
                         node[font=\large]{$\bm{F}^\top$};
\filldraw (79.375,16.875) circle (0pt)
                         node[font=\large]{$\bm{F}$};
\filldraw (79.375,79.375) circle (0pt)
                         node[font=\large]{$\bm{G}$};
\end{tikzpicture}
\caption{Structured sampling method for distance matrices proposed in \cite{lichtenberg_structured} for the experiments.}
\end{figure}

More specifically, $\E$ is sampled according to an entry-wise Bernoulli distribution with parameter $\gamma\in [0,1]$, and in each column of $\bm{F}$, $k$ entries are sampled uniformly without replacement. $\G$ is not known at all for this experiment.

Three proteins were investigated in this experiment, identified as \texttt{1PTQ}, \texttt{1AX8}, and \texttt{1UBQ}. These proteins were downloaded from the Protein Data Bank \cite{PDB}. For all three proteins, we select $m = 20$ anchors and space them uniformly throughout $[1,n]$ where $n$ is the number of atoms in the protein. In practice, domain knowledge allows for better anchor selection which can improve algorithmic performance. For each of these proteins, we set the column sample number $k\in[6,9]$, and we select the $\E$ block Bernoulli rate $\gamma\in[.1,.5]$. We additionally run this experiment with an assumed ground truth rank for $\X$ of both 3 and 4, as overparameterization has been shown previously to improve numerical performance for sensor network localization problems\cite{tang2023,lei2023}.

We first show a table indicating that increasing the rank of $\mfr$ from 3 to 4 improves the numerical performance in this structured sampling setting using Algorithm~\ref{alg:F_omega descent}. For this experiment, we will look at the protein \verb|1PTQ| ($n = 402$), and we fix the $\E$ block rate $\gamma = .3$, seen in Table~\ref{tab: Rank table}. This experiment and all others following were averaged over $25$ trials, each lasting for $10000$ iterations or until a relative difference in Frobenius norm between iterates of $10^{-5}$ was achieved.

\begin{table}[ht!]
\vspace{-0.15in}
\centering
\caption{RMSE between $\P_\mathrm{rev}$ and $\P$ averaged over $25$ trials for the protein \texttt{1PTQ} using Algorithm~\ref{alg:F_omega descent}, each trial run for $10000$ iterations or until a $10^{-5}$ relative difference in Frobenius norm is achieved.\vspace{0.05in}} \label{tab: Rank table}
\begin{tabular}{|l|l|l|l ||l  |l |l  |l|} 
\hline
   {Column samples}&  {E block rate} & {Rank}& {RMSE}& {Column samples}&  {E block rate} & {Rank}& {RMSE} \\
\hline
    6&                    .3 &           3  &  2.56&   6&                  .3 &          4  &   0.324\\
\hline
    7&                  .3 &           3  &  1.51&   7&                  .3 &          4  &   0.211\\
\hline
    8&                 .3 &           3  &  0.915&   8&                  .3 &          4  &   0.134\\
\hline
    9&                     .3 &           3  &  0.712&   9&                  .3 &          4  &   0.102\\
\hline
 \end{tabular}
\end{table}

 This experiment is in line with existing literature on overparameterization aiding reconstruction, as this provides clear indication that the reconstruction of the ground truth improves with higher rank. As such, we will set the rank of $\mfr$ to 4 for the remainder of the experiments. Next, we test to see if the $\E$ block rate parameter $\gamma$ exhibits a substantial performance effect on the final RMSE, seen in Table~\ref{tab: E block RMSE table}.

\begin{table}[ht!]
\vspace{-0.15in}
\centering
\caption{RMSE between $\P_\mathrm{rev}$ and $\P$ averaged over 25 trials for the protein \texttt{1PTQ} using Algorithm~\ref{alg:F_omega descent}, each trial run for 10000 iterations or until a $10^{-5}$ relative difference in Frobenius norm is achieved.\vspace{0.05in}} \label{tab: E block RMSE table}
\begin{tabular}{|l|l|l||l|l|l|} 
\hline
   {Column samples}&  {$\E$ block rate} & {RMSE}& {Column samples}&  {$\E$ block rate}& {RMSE} \\
\hline
6&  .1 &   0.347 & 8&   .1 &    0.142 \\
\hline
6& .2  &  0.379 & 8&   .2&  0.143\\
\hline
6&  .3&    0.324 & 8&   .3&     0.134 \\
\hline
6& .4&    0.326 & 8&   .4&    0.134\\
\hline
6&  .5&    0.329& 8&   .5&    0.121\\
\hline
\hline
7& .1 &     0.209& 9&   .1 &   0.107 \\
\hline
7&  .2&    0.191& 9&   .2&    0.104\\
\hline
7& .3&    0.211 & 9&   .3&    0.102 \\
\hline
7&  .4&    0.195& 9&   .4&    0.0954\\
\hline
7& .5&    0.200& 9&   .5&    0.0987\\\hline
 \end{tabular}
\end{table}

From the experiment in Table~\ref{tab: E block RMSE table}, increasing the $\E$ block rate does not greatly improve in the final RMSE following reconstruction. From a total number of samples perspective, this is not surprising, as for $m = 20$, the expected number of samples in the $\E$ block for $\gamma  = 0.1$ is 38, and for $\gamma = 0.5$ the expected number is 190. Given $n=402$ for this dataset, $L = 80601$, and the relative difference in total number of visible samples is less than two tenths of a percent. Since this parameter does not demonstrate a strong effect on convergence of Algorithm~\ref{alg:F_omega descent}, we will now just show the remaining experiments for the proteins \texttt{1AX8} and \texttt{1UBQ} with the $\E$ block rate $\gamma = 0.3$, seen in Table~\ref{tab: Full F_omega RMSE table}.

\begin{table}[ht!]
\centering
\caption{RMSE between $\P_\mathrm{rev}$ and $\P$ averaged over $25$ trials for the proteins \texttt{1PTQ}, \texttt{1AX8}, \texttt{1UBQ} using Algorithm~\ref{alg:F_omega descent}, each trial run for $10000$ iterations or until a $10^{-5}$ relative difference in Frobenius norm is achieved.\vspace{0.05in}} \label{tab: Full F_omega RMSE table}
\begin{tabular}{|ll|ll|ll|}
\hline
\multicolumn{2}{|c|}{\texttt{1PTQ} $(n =402)$}                                       & \multicolumn{2}{c|}{\texttt{1AX8} $(n =1003)$}                                       & \multicolumn{2}{c|}{\texttt{1UBQ} $(n =660)$}                                       \\ \hline
\multicolumn{1}{|c|}{Column Samples} & \multicolumn{1}{c|}{RMSE} & \multicolumn{1}{c|}{Column Samples} & \multicolumn{1}{c|}{RMSE} & \multicolumn{1}{c|}{Column Samples} & \multicolumn{1}{c|}{RMSE} \\ \hline
\multicolumn{1}{|l|}{6}              &          0.324& \multicolumn{1}{l|}{6}              &                           0.915& \multicolumn{1}{l|}{6}              &                           0.409\\ \hline
\multicolumn{1}{|l|}{7}              &                           0.211& \multicolumn{1}{l|}{7}              &                           0.435& \multicolumn{1}{l|}{7}              &                           0.266\\ \hline
\multicolumn{1}{|l|}{8}              &                           0.134& \multicolumn{1}{l|}{8}              &                           0.269& \multicolumn{1}{l|}{8}              &                           0.205\\ \hline
\multicolumn{1}{|l|}{9}              &                           0.102& \multicolumn{1}{l|}{9}              &                           0.201& \multicolumn{1}{l|}{9}              &                           0.177\\ \hline
\end{tabular}
\end{table}
These experiments indicate strong reconstruction ability with Algorithm $\ref{alg:F_omega descent}$ in this structured sampling setting. The dependence on the number of column samples in RMSE following reconstruction is visible as well across the experiments. This is expected from the perspective of total entries viewed in the underlying matrix, as the majority of the accessible connectivity information is stored in the $\bm{F}$ block in this sampling setup.

\begin{figure*}[t!]
    \centering
    \subfloat{\includegraphics[width=2.5in]{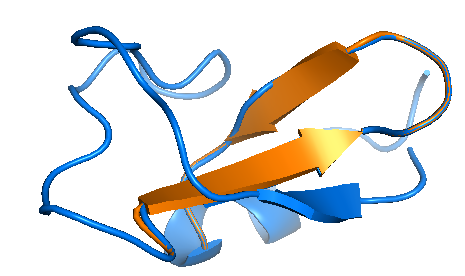}
        \label{fig:1ptq_comparisons}}
    \hfil
    \subfloat{\includegraphics[width=2.5in]{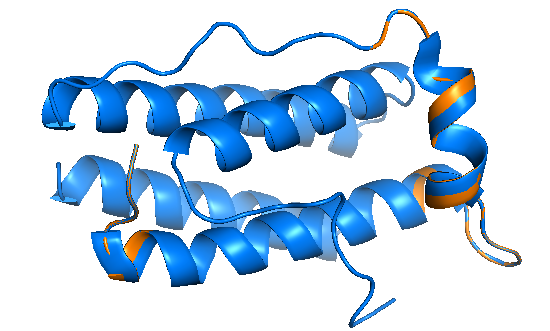}
       
        \label{fig:1ax8_comparisons}}
    \caption{Target structure (in blue) and numerically estimated structure (in orange) following $100000$ iterations of Algorithm~\ref{alg:F_omega descent}. (Left) Target structure \texttt{1AX8}, $\gamma = 0.3$ and $k=6$ (RMSE = 0.014). (Right): Target structure \texttt{1AX8}, $\gamma = 0.3$ and $k=6$ (RMSE = 0.06).}
\end{figure*}

\section{Conclusion and Future Work}\label{sec: Conclusion}

In this work we proposed a novel approach for solving the EDG problem using a matrix completion approach on the manifold of rank-$r$ matrices in Algorithm~\ref{alg:R_omega descent}. We derived local linear convergence guarantees for this non-convex Riemannian gradient-like algorithm, and with this approach we provided two provably convergent initialization techniques when considering uniformly sampled distances. To the authors' knowledge, this is the first work to provide such initialization methods non-convex approaches to the EDG problem. The convergence analysis of this algorithm was predicated on understanding properties of a non-self-adjoint sampling operator, which required novel analysis of EDG-specific bases. We provided numerical results for this method to underline its efficacy in the high-sampling regime for the EDG problem. In addition to the provably convergent Algorithm~\ref{alg:R_omega descent}, we provided an additional algorithm, Algorithm~\ref{alg:F_omega descent}, that is a true first-order method on the manifold of rank-$r$ matrices. This algorithm, although currently lacking in convergence guarantees, exhibited better numerical performance than the provably convergent one, performing nearly as well as some existing methods. Finally, we numerically investigated a structured sampling method relevant to the sensor network localization and protein structure problems, and studied how Algorithm~\ref{alg:F_omega descent} performed numerically in this setting on real-world data. We showed that it exhibited strong reconstruction performance in this new sampling framework, opening the door to future investigation.

One future goal will be a full characterization of the convergence of Algorithm~\ref{alg:F_omega descent}, as this remains an open question. This will be important to investigate due to its stronger numerical performance. We are also interested in reconstruction of matrices expanded in more general non-orthogonal bases, and developing guarantees based on linear-algebraic properties similar to those investigated in this work. Additionally, this work relied on a uniform sampling with replacement model. Oftentimes, real world models for EDG or sensor network localization rely on different sampling models, such as nearest neighbor sampling. We are interested in seeing how we can extend this work and gain theoretical guarantees in the direction of non-uniform sampling models, alongside motivating other algorithmic developments.

\section{Acknowledgment}
HanQin Cai acknowledges partial support from the National Science Foundation through grant DMS-2304489. 
Abiy Tasissa and Chandler Smith acknowledge partial support from the National Science Foundation through grant DMS-2208392. 

\bibliographystyle{IEEEtran} 
\bibliography{IEEEabrv,Arxiv_Document}

\appendices
\section{Properties of the dual bases and Non-Commutative Bernstein Inequality}\label{appendix: dual basis}
This section of the appendix details technical results about the specific dual bases, $\{\wa\}_{\ai}$ and $\{\va\}_{\ai}$. These are needed to prove various technical lemmas throughout the work, but are particularly important in the proof of Theorem~\ref{thm: R_omega RIP}. Additionally, we provide a variant of the non-commutative Bernstein inequality leveraged throughout this work.

\begin{thm}[Operator Bernstein Inequality\cite{recht2011simpler}]\label{thm:Bernstein}
Let $\X_i$, $i=1,...,m$ be i.i.d, zero-mean, matrix-valued random variables, and let $\rho_i^2\geq \max\left\{\mathbb{E}\left(\X_i\X_i^\star)\right),\mathbb{E}\left(\X_i^\star\X_i)\right)\right\}$. Assume there exists a $c\in\bb{R}$ such that $\Vert \X_i\Vert\leq c$ almost surely. Then for $t<\sum_{i=1}^m \frac{\rho_i}{c}$,
\[
\bb{P}\left(\left\Vert\sum_{i=1}^m \X_i\right\Vert>t\right)\leq 2n \exp\left(-\frac{t^2/2}{\sum_{i=1}^m \rho_i^2 + ct/3}\right).
\]
If we assume that $\rho_1^2 = ... = \rho_m^2 = V_0$ and let $V = mV_0$, then for $t<\frac{V}{c}$ this simplifies to
\begin{equation}\label{eqn:Bernstein}
    \bb{P}\left(\left\Vert\sum_{i=1}^m \X_i\right\Vert>t\right)\leq 2n \exp\left(-\frac{3t^2}{8V}\right).
\end{equation}
\end{thm}

One result that will be used throughout this work is a technique for constructing eigenvalue bounds through a vectorization technique. This result is as follows.
\begin{lem}[Vectorization Technique]\label{lem: vectorization}
    Let $\{\Z_k\}_{k=1}^m$ be a basis for some subspace $\mathbb{V}\subset\Rnn$ of dimension $m$, and let $\G = [\la\Z_i,\Z_j\ra]\in\real^{m\times m}$, and let $\Z_{\mathbb{V}}\in\real^{n^2\times m}$ be the matrix where the $k$-th column vector is $\Vec(\Z_k)$. Then for any $\Y\in\Rnn$
    \[
    \max_{\Vert\Y\Vert_\fro=1}\sum_{k=1}^m\la\Y,\Z_k\ra^2 = \lambda_\mathrm{max}(\G).
    \]
\end{lem}
\begin{proof}[Proof of Lemma~\ref{lem: vectorization}]
    We can see that
    \begin{align*}
       \max_{\Vert\Y\Vert_\fro=1}\sum_{k=1}^m \la\Y,\Z_k\ra^2  &= \max_{\Vert\Y\Vert_\fro=1}\sum_{k=1}^m \left(\Vec(\Y)^\top\Vec(\Z_k)\right)\left(\Vec(\Z_k)^\top\Vec(\Y)\right)\\
       &= \max_{\Vert\Y\Vert_\fro=1}\Vec(\Y)^\top\left(\sum_{k=1}^m \Vec(\Z_k)\Vec(\Z_k)^\top\right)\Vec(\Y)\\
       &= \max_{\Vert\Y\Vert_\fro=1}\Vec(\Y)^\top \Z_{\mathbb{V}}\Z_{\mathbb{V}}^\top\Vec(\Y).
    \end{align*}
    As for any matrix $\A\succeq\bm{0}$, $\max_{\Vert\x\Vert_2=1}\x^\top\A\x= \lambda_\mathrm{max}(\A)$, it follows that $\max_{\Vert\Y\Vert_\fro=1}\sum_{k=1}^m \la\Y,\Z_k\ra^2 = \lambda_\mathrm{max}(\Z_{\mathbb{V}}\Z_{\mathbb{V}}^\top).$
    Now, as for any $\A\in\real^{r\times s}$, $\lambda_\mathrm{max}(\A\A^\top) = \lambda_\mathrm{max}(\A^\top\A)$, we see that
    \begin{equation*}\max_{\Vert\Y\Vert_\fro=1}\sum_{k=1}^m \la\Y,\Z_k\ra^2 = \lambda_\mathrm{max}(\Z_{\mathbb{V}}\Z_{\mathbb{V}}^\top)= \lambda_\mathrm{max}(\Z_{\mathbb{V}}^\top\Z_{\mathbb{V}})
         = \lambda_\mathrm{max}(\G).
    \end{equation*}
    This concludes the proof.
\end{proof}

\begin{lem}[Spectral norm of $\Ro$]\label{lem:R_omega bound}

For $m\geq \frac{8}{3}\beta n\log(n)$ and with probability at least $1-2n^{1-\beta}$,
    \[
    \Vert\Ro\Vert\leq\frac{m}{L}+4\sqrt{\frac{8m\log(n)}{3n}}.
    \]
\end{lem}
\begin{proof}[Proof of Lemma~\ref{lem:R_omega bound}]
Compared to analogous sampling operators in matrix completion, $\Ro$ is not self-adjoint. As such, it cannot be decomposed into a sum of orthogonal projection operators. This means that the operator norm $\Vert \Ro\Vert$ cannot be bounded via a counting argument like in \cite{recht2011simpler}, as that would produce an upper bound for the maximum eigenvalue but not the maximum singular value. As such, we will proceed by using Theorem~\ref{thm:Bernstein} to prove a bound for $\left\Vert \Ro - \frac{m}{L}\mathcal{I}\right\Vert$. To do so, let
\[
\mathcal{T}_{\alphab} = \la \cdot,\wa\ra \va - \frac{1}{L}\mathcal{I}.
\]
This object is zero-mean, and $\Ro-\frac{m}{L}\mathcal{I} = \sum_{\alphab\in\Omega}\mathcal{T}_{\alphab}$. We now need bounds on $\Vert \mathcal{T}_{\alphab}\Vert$, $\Vert\mathbb{E}[\mathcal{T}_{\alphab}\mathcal{T}_{\alphab}^\star]\Vert$, and $\Vert\mathbb{E}[\mathcal{T}_{\alphab}^\star\mathcal{T}_{\alphab}]\Vert$.

For the first, notice that
\begin{align*}
    \Vert\Ta\Vert &=\left\Vert\la\cdot,\wa\ra\va - \frac{1}{L}\mathcal{I}\right\Vert\\
    &\leq \left\Vert\la\cdot,\wa\ra\va \right\Vert + \frac{1}{L}\\
    &\leq \Vert\wa\Vert_\fro\Vert\va\Vert_\fro + \frac{1}{L}\\
    &\leq \frac{2}{\sqrt{2}}+\frac{1}{L}\\
    &\leq 2 =: c,
\end{align*}
where the third to last inequality follows from Lemma~\ref{lem: H and H^-1 eigvals} and the fact that $\Vert\wa\Vert_\fro = 2$.
Next, notice that
\[
\mathbb{E}[\Ta^\star\Ta] = \frac{1}{L}\sum_{\ai} \la\cdot,\wa\ra\la\va,\va\ra\wa - \frac{1}{L^2}\mathcal{I},\qquad \mathbb{E}[\Ta\Ta^\star] = \frac{1}{L}\sum_{\ai} \la\cdot,\va\ra\la\wa,\wa\ra\va-\frac{1}{L^2}\mathcal{I}.
\]
Now, notice that
\begin{align*}
    \left\Vert \mathbb{E}[\Ta\Ta^\star]\right\Vert &= \left\Vert \frac{1}{L}\sum_{\ai} \la\cdot,\va\ra\la\wa,\wa\ra\va-\frac{1}{L^2}\mathcal{I}\right\Vert\\
    &\leq \frac{1}{L} \max_{\Vert \X\Vert_\fro = 1 }\sum_{\ai}\la\X,\va\ra^2\la\wa,\wa\ra + \frac{1}{L^2}\\
    &\leq \frac{4}{L}\max_{\Vert \X\Vert_\fro = 1 }\sum_{\ai}\la\X,\va\ra^2 + \frac{1}{L^2}\\
    &\leq \frac{4}{L}\lambda_{\mathrm{max}}(\H^{-1}) + \frac{1}{L^2}\\
    &\leq\frac{4}{L},
\end{align*}
where the first inequality follows from the triangle inequality, the second comes from $\Vert\wa\Vert_\fro = 2$ in Lemma~\ref{lem: H and H^-1 eigvals}, the third is an application of Lemma~\ref{lem: vectorization}, and the last comes from the fact that $\lambda_\mathrm{max}(\H^{-1}) = \frac{1}{2}$ from Lemma~\ref{lem: H and H^-1 eigvals}. Next, we can see that 
\begin{align*}
    \left\Vert \mathbb{E}[\Ta^\star\Ta]\right\Vert &= \left\Vert \frac{1}{L}\sum_{\ai} \la\cdot,\wa\ra\la\va,\va\ra\wa-\frac{1}{L^2}\mathcal{I}\right\Vert\\
    &\leq \max_{\Vert \X\Vert_\fro = 1 }\frac{1}{L}\sum_{\ai}\la\X,\wa\ra^2\la\va,\va\ra + \frac{1}{L^2}\\
    &\leq \frac{1}{2L}\max_{\Vert \X\Vert_\fro = 1 }\sum_{\ai}\la\X,\wa\ra^2+\frac{1}{L^2}\\
    &\leq \frac{1}{2L}\lambda_{\mathrm{max}}(\H) + \frac{1}{L^2}\\
    &\leq\frac{2n}{L},
\end{align*}
where the first inequality follows from the triangle inequality, the second comes from $\Vert\va\Vert_\fro \leq \frac{1}{\sqrt{2}}$ in Lemma~\ref{lem: H and H^-1 eigvals}, the third is an application of Lemma~\ref{lem: vectorization}, and the last comes from the fact that$\lambda_\mathrm{max}(\H)=2n$ from Lemma~\ref{lem: H and H^-1 eigvals}. As such, our variance estimate $V_0 = \frac{2n}{L}$. It follows that for any $t<\frac{mV_0}{c} = \frac{mn}{L}$, we have the following result from Theorem~\ref{thm:Bernstein}:
\begin{equation*}
 \mathbb{P}\left(\left\Vert \Ro - \frac{m}{L}\mathcal{I}\right\Vert\geq \frac{mn}{L}\sqrt{\frac{8\beta n\log(n)}{3m}}\right)\leq 2n\exp\left(-\frac{n^2 \beta\log(n)}{L}\right)
\leq 2n\exp\left(-\beta\log(n)\right) = 2n^{1-\beta},
\end{equation*}
and the proof statement follows from this.
\end{proof}

\begin{lem}[$\lambda_\mathrm{max}(\tilde{\H})$ bound]\label{lem: Bound for largest eigval of H tilde}
    Let $\tilde{\H} = [\langle \Pu\wa,\Pu\wb\rangle]\in\mathbb{R}^{L\times L}$, where $U$ is the row/column space of the true solution $\X = \U\D\U^\top$, which is rank-$r$, and where $\Pu$ is the projection operator onto $U$. It follows that
    \begin{equation*}
        \lambda_\mathrm{max}(\tilde{\H})\leq \nu r.
    \end{equation*}
\end{lem}
\begin{proof}[Proof of Lemma~\ref{lem: Bound for largest eigval of H tilde}
]
First, by coherence we have that 
\begin{equation*}
    \vert\langle \Pu\wa,\Pu\wb\rangle\vert \leq \Vert \Pu\wa\Vert_\fro \Vert\Pu\wb\Vert_\fro\leq \frac{\nu r}{2n}.
\end{equation*}
Next, as $\Pu = \U\U^\top$, for $\alphab\cap\betab=\emptyset$
\begin{equation*}
    \langle \Pu\wa,\Pu\wb\rangle = \Tr(\wa\Pu\Pu\wb)
    = \Tr(\wb\wa\Pu)
    = \Tr(\bm{0}\Pu)
    = 0,
\end{equation*}
as $\wa\wb = \wb\wa = \bm{0}$, where $\bm{0}$ is the zero matrix. Thus $\tilde{\H}$ is sparse, with each row having at most $2n-3$ non-zero entries. The result follows from a Gershgorin argument and the entrywise bound derived from the coherence condition above.
\end{proof}

\begin{lem}\label{lem: Pu wa inner product identity}
    For any $\X\in\Rnn$, $\X=\X^\top$, and any $\wa\in\{\wb\}_{\betab\in\mathbb{I}}$,
    \begin{equation*}
        \langle \Pt\X,\wa\rangle = \langle \X\Pu,\Pu\wa\rangle.
    \end{equation*}
    Additionally for $\Vert \X\Vert_\fro=1$, 
    \[
    \sum_{\ai}\la\X\Pu,\Pu\wa\ra^2\leq 
    \max_{\Vert\X\Vert_\fro=1}\sum_{\ai}\la\X\Pu,\Pu\wa\ra^2
    \leq\lambda_{\mathrm{max}}(\tilde{\H}).
    \]
\end{lem}
\begin{proof}[Proof of Lemma~\ref{lem: Pu wa inner product identity}]
    First, notice that $\langle \X\Pu,\wa\rangle = \langle \Pu\X,\wa\rangle$ due to cyclicity of the trace and symmetry of $\X,~\Pu$, and $\wa$. It follows then that
    \begin{align*}
        \langle \Pt\X,\wa\rangle &= \langle\Pu\X+\X\Pu-\Pu\X\Pu,\wa\rangle \\
        &= 2\langle\Pu\X,\wa\rangle - \langle\Pu\X\Pu,\wa\rangle\\
        &= \langle \Pu\X,\wa\rangle + \langle \Pu\X-\Pu\X\Pu,\wa\rangle\\
        &= \langle \Pu\X,\wa\rangle + \langle \Pu\X\Pup,\wa\rangle\\
        &=\langle \X,\Pu\wa\rangle + \langle \X\Pup,\Pu\wa\rangle \\
        &=\langle \X-\X\Pup,\Pu\wa\rangle \\
        &=\langle \X\Pu,\Pu\wa\rangle.
    \end{align*}
    The second statement follows from Lemma~\ref{lem: vectorization} and the fact that $\Pu$ is an orthogonal projection operator. This concludes the proof.
\end{proof}

\begin{lem}[Eigenvalues of $\H$ and $\H^{-1}$, entries of $\H^{-1}$, and spectral norms of $\wa$ and $\va$ \cite{lichtenberg2023dual}]\label{lem: H and H^-1 eigvals}
    Let $\H = [\wa,\wb]\in\real^{L\times L}$ be the Gram matrix for $\{\wa\}$, and let $\H^{-1}$ be its inverse. Then
    \[
    \lambda_\mathrm{max}(\H) = 2n, \qquad \lambda_\mathrm{max}(\H^{-1}) = \frac{1}{2}.
    \]
    Additionally, 
    \begin{equation*}
        H^{\alphab\betab} = \begin{cases}
            \frac{1}{n^2} & \alphab\cap\betab = \emptyset;\\
            -\frac{1}{2n} +\frac{1}{n^2} & \alphab\cap\betab \neq\emptyset, \alphab\neq\betab;\\
            \frac{1}{2}\left(1-\frac{1}{n}+\frac{2}{n^2}\right) & \alphab=\betab.
        \end{cases}
    \end{equation*}
    Finally,
    \begin{equation*}
        \Vert \wa\Vert = 2, \qquad \Vert\va\Vert = \frac{1}{2}.
    \end{equation*}
\end{lem}

\begin{lem}\label{lem:Form of va^2}
Let $\{\va\}_{\alphab\in\bb{I}}$ be the dual basis to $\{\wa\}_{\alphab\in\bb{I}}$. It follows that
    \begin{equation*}
        \sum_{\alphab\in\bb{I}} \va^2 = \frac{n^2-2n+2}{4n} \J .
    \end{equation*}
\end{lem}
\begin{proof}[Proof of Lemma~\ref{lem:Form of va^2}]
    Recall that $\va = -\frac{1}{2}\left(\a\b^\top + \b\a^\top\right)$ where $\a = \e_i - \frac{1}{n}\one$ and $\b = \e_j - \frac{1}{n}\one$ for $\alphab = (i,j)$. It follows that
    \begin{equation*}
        4\va^2 = \a\b^\top\a\b^\top + \a\b^\top\b\a^\top + \b\a^\top\a\b^\top + \b\a^\top\b\a^\top,
    \end{equation*}
    and as $\b^\top\b = \a^\top\a = \frac{n-1}{n}$ and $\a^\top\b = -\frac{1}{n}$, we see that
    \begin{align*}
        4\va^2 &= \frac{n-1}{n}\left[\left(\e_{ii} -\frac{1}{n}\e_i\one^\top -\frac{1}{n}\one\e_i^\top + \frac{1}{n^2}\oot\right)+\left(\e_{jj}- \frac{1}{n}\e_{j}\one^\top-\frac{1}{n}\one\e_j^\top +\frac{1}{n^2}\oot \right)\right]\\
        &\quad- \frac{1}{n}\left[\left(\e_{ij} -\frac{1}{n}\e_i\one^\top - \frac{1}{n}\one\e_j + \frac{1}{n^2}\oot \right) + \left(\e_{ji} -\frac{1}{n}\e_j\one^\top - \frac{1}{n}\one\e_i^\top + \frac{1}{n^2}\oot \right)\right]\\
        &= \frac{n-1}{n}\left(\e_{ii}+\e_{jj}\right) + \frac{2-n}{n^3}\left(\e_i\one^\top+\one\e_i^\top+\e_j\one^\top+\one\e_j^\top\right)+\frac{2(n-2)}{n^2}\oot -\frac{1}{n}\left(\e_{ij}+\e_{ji}\right).
    \end{align*}
    So it follows that
    \begin{align*}
        \sum_{\alphab\in\bb{I}}4\va^2 &= \frac{(n-1)^2}{n}\I +\frac{2(2-n)(n-1)}{n^2}\oot + \frac{(n-1)(n-2)}{n^2}\oot -\frac{1}{n}(\oot-\I),\\
        &=\frac{n^2-2n+2}{n}\I - \frac{n^2-2n+2}{n^2}\oot,
        \end{align*}
    yielding the desired result as $\J = \I - \frac{1}{n}\oot$.
\end{proof}

\section{Restricted Isometry Results}\label{appendix: RIP}
As RIP and its variants are critical to the analysis of Algorithm~\ref{alg:R_omega descent}, this section is dedicated to the proofs of RIP and similar results.

\subsection{Proof of Theorem~\ref{thm: R_omega RIP}}\begin{proof}\label{proof: R_omega RIP}
    First, notice that for any dual basis pair $\{\wa\}_{\ai}$ and $\{\va\}_{\ai}$, we can decompose any $\X\in\mathbb{S}$ as 
    \begin{equation*}
        \X = \sum_{\ai}\langle \X,\wa\rangle \va.
    \end{equation*} 
    It follows then that
    \begin{equation*}
        \mathbb{E}\left( \Ro \right) = \frac{m}{L} \mathcal{I},
    \end{equation*}
    where $\vert\Omega\vert = m$ and $\mathcal{I}$ is the identity operator on $\mathbb{S}$. From this, we can see that
    \begin{align*}
        \Pt\X = \sum_{\ai} \langle \X,\Pt\wa\rangle\va,\qquad
        \Ro\Pt\X = \sum_{\alphab\in\Omega} \langle \X,\Pt\wa\rangle\va, \qquad
        \Pt\Ro\Pt\X = \sum_{\alphab\in\Omega} \langle \X,\Pt\wa\rangle\Pt\va .
    \end{align*}
    Therefore it follows that $\mathbb{E}(\Pt\Ro\Pt) = \frac{m}{L}\Pt$. We can now use Theorem~\ref{thm:Bernstein} to bound the probability that $\Pt\Ro\Pt$ deviates from its expected spectral norm. To do this, we first define an operator $\mathcal{T}_{\alphab} = \langle \cdot,\Pt\wa\rangle \Pt\va - \frac{1}{L}\Pt$.  
    First, observe the following  coherence conditions outlined in Assumption~\ref{asp: Incoherence assumption},
    \begin{align*}
        \left\Vert \la\cdot,\Pt\wa\ra\Pt\va\right\Vert &\leq \left\Vert \Pt\wa\right\Vert_\fro\left\Vert\Pt\va\right\Vert_\fro
        \leq \sqrt{\frac{\nu r}{8n}}\sqrt{\frac{\nu r}{2n}}
        \leq \frac{\nu r}{2n}.
    \end{align*}
    Additionally, $\mathbb{E}(\Ta) = \frac{1}{L}\mathcal{I}$. It follows then that
    \begin{align*}
        \left\Vert \Ta \right\Vert &\leq \left\Vert \la\cdot,\Pt\wa\ra\Pt\va\right\Vert + \frac{1}{L} \\
        &\leq \frac{\nu r}{2n} + \frac{1}{L}\\
        &\leq \frac{\nu r}{n} 
        \leq \frac{\nu^2 r^2}{n} =: c,
    \end{align*}
    as $\nu,r\geq1$, which gives us our almost sure estimate on the spectral norm of each term $\Ta$. Next, to estimate the variance we notice first that
    \begin{align*}
        \mathbb{E}\left(\mathcal{T}_{\alphab}\mathcal{T}_{\alphab}^\star\right) =\frac{1}{L}\sum_{\ai}\Ta\Ta^\star = \frac{1}{L}\sum_{\ai} \langle \cdot,\Pt\va\rangle\langle\Pt\wa,\Pt\wa\rangle\Pt\va -\frac{1}{L^2}\Pt,\\
        \mathbb{E}\left(\mathcal{T}^\star_{\alphab}\mathcal{T}_{\alphab}\right)= \frac{1}{L}\sum_{\ai}\Ta^\star\Ta = \frac{1}{L}\sum_{\ai} \langle \cdot,\Pt\wa\rangle\langle\Pt\va,\Pt\va\rangle\Pt\wa - \frac{1}{L^2}\Pt.
    \end{align*}
        To bound the maximum spectral norm of the above two terms, notice that for $\sum_{\ai} \mathcal{T}_{\alphab}\mathcal{T}_{\alphab}^\star$,
        \begin{align*}
        \left\Vert\frac{1}{L}\sum_{\ai}\mathcal{T}_{\alphab}\mathcal{T}_{\alphab}^\star\right\Vert &\leq \max_{\Vert\X\Vert_\fro=1}\frac{1}{L}\sum_{\ai} \la\Pt\wa,\Pt\wa\ra\langle \X,\Pt\va\rangle^2 + \frac{1}{L^2}\\
        &\leq \max_{\Vert\X\Vert_\fro=1}\frac{\nu r}{2nL}\sum_{\ai} \langle \X,\Pt\va\rangle^2 + \frac{1}{L^2}\\
        &= \max_{\X\in\T,\Vert\X\Vert_\fro=1}\frac{\nu r}{2nL}\sum_{\ai} \langle \X,\va\rangle^2 + \frac{1}{L^2}\\
        &\leq \max_{\Vert\X\Vert_\fro=1}\frac{\nu r}{2nL}\sum_{\ai} \langle \X,\va\rangle^2 + \frac{1}{L^2}\\
        &= \frac{\nu r}{2nL} \lambda_{\mathrm{max}}(\H^{-1}) + \frac{1}{L^2}\\
        &= \frac{\nu r}{4nL},
        \end{align*}
        where the first inequality follows from the triangle inequality, the second comes from the coherence conditions in Assumption~\ref{asp: Incoherence assumption}, the third line comes from the self-adjointness of $\Pt$, the fourth comes from the definition of the max, the fifth comes from an application of Lemma~\ref{lem: vectorization}, and the sixth comes from the $\lambda_\mathrm{max}(\H^{-1})$ bound from Lemma~\ref{lem: H and H^-1 eigvals}.
    Next for $\sum_{\ai} \mathcal{T}^\star_{\alphab}\mathcal{T}_{\alphab}$,
    we have that
    \begin{align*}
        \left\Vert\frac{1}{L}\sum_{\ai}\mathcal{T}^\star_{\alphab}\mathcal{T}_{\alphab}\right\Vert &\leq \max_{\Vert\X\Vert_\fro=1}\frac{1}{L}\sum_{\ai}  \la \Pt\va,\Pt\va\ra\la\X,\Pt\wa\rangle^2 + \frac{1}{L^2}\\
        &\leq \max_{\Vert\X\Vert_\fro=1}\frac{\nu r}{2nL}\sum_{\ai} \langle \X,\Pt\wa\rangle^2 + \frac{1}{L^2}\\
        &= \max_{\Z = \X\Pu,\Vert\X\Vert_\fro=1}\frac{\nu r}{2nL}\sum_{\ai} \langle \Z,\Pu\wa\rangle^2 + \frac{1}{L^2}\\
        &\leq \max_{\Vert\X\Vert_\fro=1}\frac{\nu r}{2nL}\sum_{\ai} \langle \X,\Pu\wa\rangle^2 + \frac{1}{L^2}\\
        &= \frac{\nu r}{2nL} \lambda_{\max}(\tilde{\H})+\frac{1}{L^2}\\
        &\leq \frac{\nu^2 r^2}{2nL},
    \end{align*}
    where the first inequality follows from the triangle inequality, the second comes from the coherence conditions in Assumption~\ref{asp: Incoherence assumption}, the third line comes from the fact that for any symmetric $\Y\in\Rnn$, $\langle \Pt\Y,\wa\rangle = \langle \Y\Pu,\Pu\wa\rangle$ from Lemma~\ref{lem: Pu wa inner product identity}, the fourth comes from the definition of the max, the fifth comes from an application of Lemma~\ref{lem: vectorization}, and the sixth comes from the bound on $\lambda_{\max}(\tilde{\H})$ in Lemma~\ref{lem: Bound for largest eigval of H tilde}.
    
    As the latter term is larger, we get a variance estimate $V_0 = \frac{\nu^2 r^2}{nL}$. Now, for $t<\frac{mV_0}{c} = \frac{m}{L}$, we can use \eqref{eqn:Bernstein}. It follows that for $m\geq \frac{8}{3}\beta\nu^2 r^2 n\log(n)$,
    \begin{align*}
        \mathbb{P}\left(\left\Vert \Pt\Ro\Pt-\frac{m}{L}\Pt\right\Vert \geq \frac{m}{L}\sqrt{\frac{8\beta\nu^2 r^2 n\log(n)}{3m}}\right)\leq 2n\exp\left(-\beta\log(n)\right)
        =2n^{1-\beta},
    \end{align*}
    yielding the desired result.

    Additionally, as $\mathbb{E}(\Ro^\star) = \frac{m}{L}\mathcal{I}$, the same proof can be repeated to show RIP for $\Ro^\star$ with the same constants provided.
\end{proof}

The proof of RIP for $\Ro$ allows us to define a neighborhood around the ground truth solution where a slightly weakened version of RIP holds. First, however, we will introduce some technical lemmas:
\begin{lem}[Spectral norm bounds for $\Pt\Ro$ and $\Ro\Pt$]\label{lem: Bound on PtRo and RoPt}
    Let $\X$ be a $\nu$-incoherent rank-$r$ ground truth matrix. For $m\geq\frac{16}{3}\nu r n\log(n)$, both results hold, each with probability $1-2n^{1-\beta}$:
    \[
    \Vert \Ro\Pt\Vert \leq \frac{m}{L} + \frac{m\sqrt{n}}{L}\sqrt{\frac{8\beta\nu r n\log(n)}{3m}} ~~\mathrm{and}~~\Vert \Pt\Ro\Vert\leq \frac{m}{L}+\frac{4m\sqrt{n}}{L}\sqrt{\frac{\beta \nu r n\log(n)}{3m}}.
    \]

\end{lem}
\begin{proof}[Proof of~\ref{lem: Bound on PtRo and RoPt}]
    We will start by proving the bound for $\Vert \Ro\Pt\Vert$. 

     First, notice that
    \[
    \Vert\Ro\Pt\Vert \leq \left\Vert \Ro\Pt-\frac{m}{L}\Pt\right\Vert + \frac{m}{L},
    \]
    following from the triangle inequality, and notice that the middle term can be decomposed as the sum of i.i.d. operators as follows. Let $\Ta = \la \cdot,\Pt\wa\ra\va-\frac{1}{L}\Pt$. Much in the same vein as in Theorem~\ref{thm: R_omega RIP} and Lemmas~\ref{lem: Asymmetric RIP} and \ref{lem:Initialization}, we will use Theorem~\ref{thm:Bernstein} to prove a concentration result. For this, we must get a spectral norm and variance estimate.

    For the spectral norm estimate, notice that
    \begin{align*}
        \Vert\Ta\Vert &= \left\Vert\la\cdot,\Pt\wa\ra\va - \frac{1}{L}\Pt\right\Vert\\
        &\leq \Vert\Pt\wa\Vert_\fro\Vert\va\Vert_\fro + \frac{1}{L}\\
        &\leq \sqrt{\frac{\nu r}{2n}} + \frac{1}{L}\\
        &\leq \frac{2\nu r}{\sqrt{n}} =: c.
    \end{align*}
    To get the variance bounds, notice that
       \begin{align*}
        \Vert \E[\Ta\Ta^\star]\Vert &= \left\Vert \frac{1}{L}\sum_{\ai}\la\cdot,\va\ra\la\Pt\wa,\Pt\wa\ra\va - \frac{1}{L^2}\Pt\right\Vert\\
        &\leq\frac{\nu r}{2nL}\lambda_{\mathrm{max}}(\H^{-1}) + \frac{1}{L^2}\\
        &\leq \frac{\nu r}{nL},
    \end{align*}
    where the first inequality follows from Assumption~\ref{asp: Incoherence assumption}, the triangle inequality, and Lemma~\ref{lem: vectorization}. 
    For the other term, we see that
    \begin{align*}
        \left\Vert \E[\Ta^\star\Ta]\right\Vert &= \left\Vert \frac{1}{L}\sum_{\ai}\la\cdot,\Pt\wa\ra\la\va,\va\ra\Pt\wa - \frac{1}{L^2}\Pt\right\Vert\\
        &\leq \left\Vert \frac{1}{L}\sum_{\ai}\la\cdot,\Pt\wa\ra\la\va,\va\ra\Pt\wa\right\Vert + \frac{1}{L^2}\\
        &\leq \frac{1}{L^2} + \frac{1}{2L}\max_{\Vert\X\Vert_\fro=1} \sum_{\ai} \la \X,\Pt\wa\ra^2\\
        &\leq \frac{1}{L^2} + \frac{1}{2L}\max_{\Vert\X\Vert_\fro=1} \sum_{\ai} \la \X,\Pu\wa\ra^2\\
        &\leq \frac{1}{L^2} + \frac{1}{2L}\lambda_{\mathrm{max}}(\tilde{\H})\\
        &\leq \frac{1}{L^2} + \frac{\nu r }{2L}\\
        &\leq \frac{\nu r}{L},
    \end{align*}
    where the third inequality follows from $\Vert\va\Vert_\fro<1$, the fourth inequality follows from Lemma~\ref{lem: Pu wa inner product identity}, and fifth inequality from Lemma~\ref{lem: Bound for largest eigval of H tilde}. As this latter term is the larger of the two variance terms, we set $V = \frac{\nu r m}{L}$ as our variance estimate. This allows us to state the following result using Theorem~\ref{thm:Bernstein} for $m\geq \frac{8}{3}\nu r \beta n\log(n)$:
    \begin{align*}
        \mathbb{P}\left(\left\Vert \Ro\Pt - \frac{m}{L}\Pt\right\Vert \geq \frac{m\sqrt{n}}{L}\sqrt{\frac{8 \beta\nu r n\log(n)}{3m}}\right)\leq 2n\exp\left(-\beta \log(n)\right)= 2n^{1-\beta},
    \end{align*}
    giving a bound on $\Ro\Pt$ of 
    \[
    \Vert \Ro\Pt\Vert \leq \frac{m}{L} + \frac{m\sqrt{n}}{L}\sqrt{\frac{8\beta\nu r n\log(n)}{3m}},
    \]
    with probability at least $1-2n^{1-\beta}$.

   The next step is to produce a similar bound for $\Vert \Pt\Ro\Vert$. The analysis is much the same, with $c = \frac{2\nu r}{\sqrt{n}}$ and $V = \frac{2 \nu r}{L}$. Following the same steps, we see that  for $m\geq \frac{16}{3}\nu r \beta n\log(n)$,
    \[
    \Vert \Pt\Ro\Vert\leq \frac{m}{L}+\frac{4m\sqrt{n}}{L}\sqrt{\frac{\beta \nu r n\log(n)}{3m}},
    \]
    with probability at least $1-2n^{1-\beta}$.
\end{proof}
    \begin{lem}[Spectral Bound on $\Ptl\Ro$]\label{lem: Ptl R_o bound}
        Assume that
            \[
            \Vert\Ro\Vert\leq\frac{m}{L}+4\sqrt{\frac{8m\log(n)}{3n}}~~\mathrm{and}~~
            \Vert \Pt\Ro\Vert \leq \frac{m}{L} + \frac{4m\sqrt{n}}{L}\sqrt{\frac{\beta\nu r n\log(n)}{3m}},
            \]
        Then 
        \[
        \Vert\Ptl\Ro\Vert\leq \left(\frac{m}{L} + 4\sqrt{\frac{8m\log(n)}{n}}\right)\frac{2\Vert \X_l - \X\Vert_\fro}{\sigma_\mathrm{min}(\X)} + \frac{m}{L}+\frac{4m\sqrt{n}}{L}\sqrt{\frac{\beta \nu r n\log(n)}{3m}} .
        \]
    \end{lem}
    \begin{proof}[Proof of~\ref{lem: Ptl R_o bound}]
    This result follows from direct computation, as
    \begin{align*}
        \Vert\Ptl\Ro\Vert&= \Vert \Ptl\Ro - \Pt\Ro + \Pt\Ro\Vert\\
       & = \Vert(\Ptl-\Pt)\Ro + \Pt\Ro\Vert\\
        &\leq\Vert\Ro\Vert\Vert\Ptl-\Pt\Vert + \Vert\Pt\Ro\Vert\\
        &\leq \left(\frac{m}{L} + 4\sqrt{\frac{8m\log(n)}{n}}\right)\frac{2\Vert \X_l - \X\Vert_\fro}{\sigma_\mathrm{min}(\X)} + \frac{m}{L}+\frac{4m\sqrt{n}}{L}\sqrt{\frac{\beta \nu r n\log(n)}{3m}} ,
    \end{align*}
    where the second inequality follows from Lemma~\ref{lem: Projection Bounds} and the assumptions of this lemma. This concludes the proof.
    \end{proof}

\begin{lem}[RIP in a Local Neighborhood]\label{lem: New Local Pseudo proof}
Assume 
\begin{gather}
    \left\Vert\frac{L}{m}\Pt\Ro\Pt-\Pt\right\Vert\leq\varepsilon_0<1\label{asp: RIP},
    \\
    \frac{\Vert\X_l-\X\Vert_\fro}{\sigma_\text{min}(\X)}\leq\frac{\sqrt{m}\varepsilon_0}{16n^{5/4}\sqrt{\beta\nu r\log{n}}}\label{asp: nbd},~~ \Vert\Ro\Vert\leq \frac{m}{L} + 4\sqrt{\frac{8m\log(n)}{n}},\\
    \Vert \Ro\Pt\Vert \leq \frac{m}{L} + \frac{m\sqrt{n}}{L}\sqrt{\frac{8\beta\nu r n\log(n)}{3m}}, ~~\mathrm{and}~~\Vert \Pt\Ro\Vert\leq \frac{m}{L}+\frac{4m\sqrt{n}}{L}\sqrt{\frac{\beta \nu r n\log(n)}{3m}}.
    \label{asp: norm bds}
\end{gather}
Then
\[
\left\Vert\Ptl - \frac{L}{m}\Ptl\Ro\Ptl\right\Vert\leq 4\varepsilon_0.
\]
\end{lem}
\begin{proof}
    First, notice that
    \begin{align*}
        \left\Vert \Ptl - \frac{L}{m}\Ptl\Ro\Ptl\right\Vert &\leq \Vert \Ptl-\Pt\Vert + \frac{L}{m}\Vert \Ptl\Ro\Ptl - \Ptl\Ro\Pt\Vert \\
        &\quad+ \frac{L}{m}\Vert \Ptl\Ro\Pt - \Pt\Ro\Pt\Vert + \left\Vert\Pt-\frac{L}{m}\Pt\Ro\Pt\right\Vert \\
        &\leq \Vert \Ptl - \Pt\Vert + \frac{L}{m}\Vert \Ptl\Ro\Vert\Vert\Ptl-\Pt\Vert \\
        &\quad+\frac{L}{m}\Vert\Ro\Pt\Vert \Vert \Ptl-\Pt\Vert + \left\Vert \Pt - \frac{L}{m}\Pt\Ro\Pt\right\Vert \\
       & \leq \frac{2\Vert \X_l-\X\Vert_\fro}{\sigma_\text{min}(\X)}\left(1 + \frac{L}{m}\Vert \Ptl\Ro\Vert +\frac{L}{m}\Vert \Ro\Pt\Vert\right) + \left\Vert\Pt-\frac{L}{m}\Pt\Ro\Pt\right\Vert \label{eqn: Last equation local PI proof},
    \end{align*}
    using the triangle inequality and the results gathered in Lemma~\ref{lem: Projection Bounds}. 

    We can now bound each of these terms using the assumptions and prior lemmas. First, notice that
    \begin{align*}
        \frac{2L\Vert \X_l-\X\Vert_\fro}{m\sigma_\mathrm{min}(\X)}\Vert \Ro\Pt\Vert &\leq \frac{2\Vert \X_l-\X\Vert_\fro}{\sigma_\mathrm{min}(\X)} + \sqrt{\frac{32\beta \nu r n^2 \log(n)}{3m}}\frac{\Vert \X_l-\X\Vert_\fro}{\sigma_\mathrm{min}(\X)}\\
        &\leq \frac{\sqrt{m}\varepsilon_0}{8n^{5/4}\sqrt{\beta\nu r\log{n}}} + \sqrt{\frac{\beta \nu r n^2 \log(n)}{24m}}\frac{\sqrt{m}\varepsilon_0}{n^{5/4}\sqrt{\beta\nu r\log{n}}} \\
        &\leq \frac{\sqrt{m}}{n^{5/4}}\frac{\varepsilon_0}{8} + \frac{\varepsilon_0}{n^{1/4}\sqrt{24}}\\
        &\leq \frac{\sqrt{L}}{n^{5/4}}\frac{\varepsilon_0}{8}+\frac{\varepsilon_0}{\sqrt{24}}\\
        &\leq \frac{\varepsilon_0}{8} + \frac{\varepsilon_0}{\sqrt{24}}\\
        &\leq \varepsilon_0,
    \end{align*}
    where the first inequality comes from the assumption on $\Vert\Ro\Pt\Vert$ in \eqref{asp: norm bds}, the second inequality comes from the local neighborhood assumption in \eqref{asp: nbd}, the third inequality comes from term cancellation and the fact that $\beta,\nu,r,\log(n)\geq 1$, the fourth inequality comes from the fact that $m\leq L$, the fifth inequality comes from the fact that $\frac{L}{n^2}<1$, and the last inequality is a numerical inequality on the fractions.

    Next, notice that the conditions of Lemma~\ref{lem: Ptl R_o bound} are satisfied, so
    \begin{align*}
        &\quad~\frac{2L\Vert \X_l-\X\Vert_\fro}{m\sigma_\mathrm{min}(\X)}\Vert \Ptl\Ro\Vert \\
        &\leq \left(\frac{2L\Vert \X_l-\X\Vert_\fro}{m\sigma_\mathrm{min}(\X)}\right)\left(\left(\frac{m}{L} + 4\sqrt{\frac{8m\log(n)}{n}}\right)\frac{2\Vert \X_l - \X\Vert_\fro}{\sigma_\mathrm{min}(\X)} + \frac{m}{L}+\frac{4m\sqrt{n}}{L}\sqrt{\frac{\beta \nu r n\log(n)}{3m}} \right)\\
        &\leq \left(\frac{2\Vert \X_l-\X\Vert_\fro}{\sigma_\mathrm{min}(\X)}\right)^2\left(1+\frac{2n^2}{m}\sqrt{\frac{8m\log(n)}{n}}\right) + \frac{2\Vert \X_l-\X\Vert_\fro}{\sigma_\mathrm{min}(\X)}\left(1+4\sqrt{\frac{\beta\nu r n^2\log(n)}{3m}}\right)\\
       & \leq \frac{m\varepsilon_0^2}{64n^{5/2}\beta\nu r\log{n}} + \frac{\sqrt{m}\varepsilon_0}{8n^{5/4}\sqrt{\beta\nu r\log{n}}} +\left(\frac{m\varepsilon_0^2}{64n^{5/2}\beta\nu r\log{n}}\right)\left(\frac{2n^2}{m}\sqrt{\frac{8m\log(n)}{n}}\right)\\
       &\quad+\left(\frac{\sqrt{m}\varepsilon_0}{8n^{5/4}\sqrt{\beta\nu r\log{n}}}\right)\left(4\sqrt{\frac{\beta\nu r n^2\log(n)}{3m}}\right)\\
       & \leq \frac{\varepsilon_0^2}{64} + \frac{\varepsilon_0}{8} + \left(\frac{m\varepsilon_0^2}{64n^{5/2}\beta\nu r\log{n}}\right)\left(\frac{2n^2}{m}\sqrt{\frac{8m\log(n)}{n}}\right)+\left(\frac{\sqrt{m}\varepsilon_0}{8n^{5/4}\sqrt{\beta\nu r\log{n}}}\right)\left(4\sqrt{\frac{\beta\nu r n^2\log(n)}{3m}}\right)\\
        &= \frac{\varepsilon_0^2}{64} + \frac{\varepsilon_0}{8} + \frac{\sqrt{32m}\varepsilon_0^2}{64\beta\nu r n\log(n)} + \frac{\varepsilon_0}{n^{1/4}\sqrt{12}}\\
        &\leq \frac{\varepsilon_0^2}{64} + \frac{\varepsilon_0}{8} + \frac{\varepsilon_0^2}{16}+ \frac{\varepsilon_0}{\sqrt{12}} \\
        &\leq \varepsilon_0,
    \end{align*}
    where the first inequality follows from the assumptions on $\Vert \Pt\Ro\Vert$ in \eqref{asp: norm bds}, the second inequality follows from rearrangement of terms and the fact that $\frac{L}{m}\leq \frac{n^2}{m}$, the third inequality comes from the local neighborhood assumption in \eqref{asp: nbd}, the fourth inequality comes from the fact that $\frac{m}{n^2}\leq\frac{L}{n^2}<1$ along with $\beta,\nu, r,\log(n)\geq 1$, the fifth line comes from multiplying out terms, the sixth inequality again comes from a bound on $\frac{m}{n^2}$ amongst other simplifications, and the last line comes from a numerical inequality about the fractions coupled with the fact that $\varepsilon_0<1$ from \eqref{asp: RIP}, so $\varepsilon_0^2\leq\varepsilon_0$.
    The desired statement follows from here, thus concluding the proof.
    \end{proof}

For Algorithm~\ref{alg:Resampling Algorithm}, we will need what the authors of \cite{wei2020guarantees} call an asymmetric form of RIP for $\Ro$. The statement and proof of this are below:
\begin{lem}[Asymmetric RIP of $\Ro$]\label{lem: Asymmetric RIP}
    Let $\X_l = \U_l\D_l\U_l^\top$ and $\X = \U\D\U^\top$ be two fixed rank-$r$ matrices. Assume 
    \begin{equation*}
        \Vert \Pu\wa\Vert_\fro \leq \sqrt{\frac{\nu r}{2n}}, \qquad\Vert \Pu\va\Vert_\fro \leq \sqrt{\frac{\nu r}{2n}},\qquad \Vert\mathcal{P}_{U_l}\wa\Vert_\fro \leq \sqrt{\frac{\nu r}{2n}},\quad\mathrm{and}\quad \Vert\mathcal{P}_{U_l}\va\Vert_\fro \leq \sqrt{\frac{\nu r}{2n}}, 
    \end{equation*}
    for $\ai$. Let $\vert \Omega\vert = m$. For $m\geq \frac{4}{3}\beta\nu^2 r^2 n\log(n)$, with probability at least $1-2n^{1-\beta}$ for $\beta>1$, the following estimate holds:
    \begin{equation}
        \left\Vert \frac{L}{m}\Ptl\Ro(\Pu-\Pul) - \Ptl(\Pu-\Pul)  \right\Vert\leq\sqrt{\frac{4\beta \nu^2 r^2 n\log(n)}{3m}}.
    \end{equation}
\end{lem}
\begin{proof}[Proof of Lemma~\ref{lem: Asymmetric RIP}]
 First, note that for all $\Z\in\Rnn$,
 \begin{align*}
     (\Pu-\Pul)\Z = \sum_{\ai} \la (\Pu-\Pul)\Z,\wa\ra\va \\
     = \sum_{\ai} \la \Z,(\Pu-\Pul)\wa\ra\va, \\
 \end{align*}
 so it follows that 
 \begin{equation*}
    \Ro(\Pu-\Pul)\Z = \sum_{\alphab\in\Omega} \la \Z,(\Pu-\Pul)\wa\ra\va,
 \end{equation*}
 and subsequently
  \begin{equation*}
    \Ptl\Ro(\Pu-\Pul)\Z = \sum_{\alphab\in\Omega} \la \Z,(\Pu-\Pul)\wa\ra\Ptl\va.
 \end{equation*}
 We define $\mathcal{K}_{\alphab}(\cdot) = \Ptl(\va)\otimes(\Pu-\Pul)(\wa)(\cdot) = \la \cdot,(\Pu-\Pul)\wa\ra\Ptl\va$. It follows that $\Ptl\Ro(\Pu-\Pul) = \sum_{\alphab\in\Omega}\mathcal{K}_{\alphab}$, and that $\mathcal{K}_{\alphab}^\star = (\Pu-\Pul)(\wa)\otimes\Ptl(\va)(\cdot) = \la \cdot,\Ptl\va\ra(\Pu-\Pul)\wa$. Additionally, note that $\mathbb{E}[\mathcal{K}_{\alphab}] = \frac{1}{L}\Ptl(\Pu-\Pul)$ and that $\mathbb{E}[\mathcal{K}_{\alphab}^\star] = \frac{1}{L}(\Pu-\Pul)\Ptl$. We note that the variance term $\mathbb{E}\left[\left(\mathcal{K}_{\alphab}-\frac{1}{L}\Ptl(\Pu-\Pul)\right)\left(\mathcal{K}_{\alphab}-\frac{1}{L}\Ptl(\Pu-\Pul)\right)^\star\right]$ can be bounded as follows:
 \begin{align*}
     &\quad~\mathbb{E}\left[\left(\mathcal{K}_{\alphab}-\frac{1}{L}\Ptl(\Pu-\Pul)\right)\left(\mathcal{K}_{\alphab}-\frac{1}{L}\Ptl(\Pu-\Pul)\right)^\star\right]  \\
     &= \mathbb{E}\left[\mathcal{K}_{\alphab}\mathcal{K}_{\alphab}^\star - \mathcal{K}_{\alphab}\frac{1}{L}(\Pu-\Pul)\Ptl - \frac{1}{L}\Ptl(\Pu-\Pul)\mathcal{K}_{\alphab}^\star + \frac{1}{L^2}\Ptl(\Pu-\Pul)\Ptl\right]\\
     &= \mathbb{E}\left[\mathcal{K}_{\alphab}\mathcal{K}_{\alphab}^\star \right] -\mathbb{E}\left[\mathcal{K}_{\alphab}\right]\frac{1}{L}(\Pu-\Pul)\Ptl- \frac{1}{L}\Ptl(\Pu-\Pul)\mathbb{E}\left[\mathcal{K}_{\alphab}^\star\right] + \frac{1}{L^2}\Ptl(\Pu-\Pul)\Ptl\\
     &=\mathbb{E}\left[\mathcal{K}_{\alphab}\mathcal{K}_{\alphab}^\star \right] -\frac{1}{L^2}\Ptl\left(\Pu-\Pul\right)^2\Ptl - \frac{1}{L^2}\Ptl\left(\Pu-\Pul\right)^2\Ptl + \frac{1}{L^2}\Ptl\left(\Pu-\Pul\right)^2\Ptl\\
     &= \mathbb{E}\left[\mathcal{K}_{\alphab}\mathcal{K}_{\alphab}^\star \right] - \frac{1}{L^2}\Ptl\left(\Pu-\Pul\right)^2\Ptl.
 \end{align*}
 A similar computation holds for $\mathbb{E}\left[\left(\mathcal{K}_{\alphab}-\frac{1}{L}\Ptl(\Pu-\Pul)\right)^\star\left(\mathcal{K}_{\alphab}-\frac{1}{L}\Ptl(\Pu-\Pul)\right)\right]$, indicating that it suffices to compute an upper bound on the following terms in order to leverage Theorem~\ref{thm:Bernstein}:
 \begin{gather*}
     \left\Vert \mathcal{K}_{\alphab}- \frac{1}{L}\Ptl(\Pu-\Pul)\right\Vert \leq c,\\
     \max\left\{\left\Vert \mathbb{E}[\Ka\Ka^\star] - \frac{1}{L^2}\Ptl(\Pu-\Pul)^2\Ptl\right\Vert, \left\Vert \mathbb{E}[\Ka^\star\Ka] - \frac{1}{L^2}(\Pu-\Pul)\Ptl(\Pu-\Pul)\right\Vert\right\}\leq V_0.
 \end{gather*}

For the first term, notice that
\begin{align*}
    \Vert \Ka\Vert &\leq \Vert(\Pu-\Pul)\wa\Vert_\fro\Vert\Ptl\va\Vert_\fro\\
    &\leq (\Vert\Pu\wa\Vert_\fro+\Vert\Pul\wa\Vert_\fro)\Vert \Ptl\va\Vert_\fro\\
    &\leq \frac{\nu r}{n} \leq \frac{\nu^2 r^2}{n} .
\end{align*}
So by the triangle inequality $\left\Vert \Ka - \frac{1}{L}\Ptl(\Pu-\Pul)\right\Vert\leq \frac{2\nu^2 r^2}{n} =: c$.
For the second term, notice that 
\begin{align*}
    \mathbb{E}[\Ka\Ka^\star] = \frac{1}{L}\sum_{\ai}\langle\cdot,(\Pu-\Pul)\wa\rangle\la\Ptl\va,\Ptl\va\ra(\Pu-\Pul)\wa,\\
        \mathbb{E}[\Ka^\star\Ka] = \frac{1}{L}\sum_{\ai}\langle\cdot,\Ptl\va\rangle\la(\Pu-\Pul)\wa,(\Pu-\Pul)\wa\ra\Ptl\va.
    \end{align*}
As such,
\begin{align*}
    \Vert\mathbb{E}[\Ka\Ka^\star]\Vert &= \frac{1}{L}\max_{\X\in\Rnn,\Vert\X\Vert_\fro=1}\sum_{\ai}\la\X,(\Pu-\Pul)\wa\ra^2\la\Ptl\va,\Ptl\va\ra\\
    &\leq \frac{\nu r}{2nL}\max_{\X\in\Rnn,\Vert\X\Vert_\fro=1} \sum_{\ai}\la\X,\underbrace{(\Pu-\Pul)\wa}_{=:\tilde{\bm{w}}^l_{\alphab}}\ra^2\\ 
    &= \frac{\nu r}{2nL}\lambda_{\mathrm{max}}(\tilde{\H}^l),
\end{align*}
where $\tilde{\H}^l = [\la \tilde{\bm{w}}^l_{\alphab},\tilde{\bm{w}}^l_{\betab}\ra]\in\mathbb{R}^{L\times L}$.  To bound this, notice that if $\alphab\cap\betab = \emptyset$,
\begin{align*}
    \la \tilde{\bm{w}}^l_{\alphab},\tilde{\bm{w}}^l_{\betab}\ra &= \la \wa,(\Pu-\Pul)^2\wb\ra\\
    &=\Tr\left[\wa(\Pu-\Pul)^2\wb\right]\\
    &=\Tr\left[\wb\wa(\Pu-\Pul)^2\right]\\
    &= 0,
\end{align*}
therefore preserving the same sparsity structure as $\tilde{\H}$. To bound the magnitude of the entries, we can see that 
\begin{align*}
    \vert\la \tilde{\bm{w}}^l_{\alphab},\tilde{\bm{w}}^l_{\betab}\ra\vert& = \vert\la \wa,(\Pu-\Pul)^2\wb\ra\vert
    = \vert \la\wa,(\Pu - \Pu\Pul-\Pul\Pu+\Pul)\wb\ra\vert\\
    &\leq \vert\la \Pu\wa,\Pu\wb\ra\vert + \vert\la\Pu\wa,\Pul\wb\ra\vert + \vert\la\Pul\wa,\Pu\wb\ra\vert+\vert\la\Pul\wa,\Pul\wb\ra\vert\\
    &\leq \frac{2\nu r}{n},
\end{align*}
giving us an upper bound on $\tilde{\H}^l$ of
\begin{equation*}
    \lambda_\mathrm{max}(\tilde{\H}^l) \leq \frac{2\nu r}{n}(2n-3)\leq 4\nu r.
\end{equation*}
This gives a bound of
\begin{equation*}
    \Vert\mathbb{E}[\Ka\Ka^\star]\Vert \leq \frac{2\nu^2 r^2}{nL}.
\end{equation*}
Next, we can see that
\begin{align*}
    \Vert\mathbb{E}[\Ka^\star\Ka]\Vert &= \frac{1}{L}\max_{\X\in\Rnn,\Vert\X\Vert_\fro=1} \sum_{\ai}\left\la (\Pu-\Pul)\wa,(\Pu-\Pul)\wa\right\ra\la\X,\Ptl\va\ra^2\\
    &\leq \frac{2\nu r}{nL}\max_{\X\in\T,\Vert\X\Vert_\fro=1}\sum_{\ai}\la\X,\va\ra^2\\
    &\leq \frac{2\nu r}{nL}\lambda_{\mathrm{max}}(\H^{-1})\\
    &=\frac{\nu r}{nL}.
\end{align*}
Now, it follows that
\begin{align*}
    &\left\Vert \mathbb{E}[\Ka\Ka^\star] - \frac{1}{L^2}\Ptl(\Pu-\Pul)^2\Ptl\right\Vert \leq \frac{2\nu^2 r^2}{nL} + \frac{4}{L^2}\leq\frac{4 \nu^2 r^2}{nL}\\
    &\left\Vert \mathbb{E}[\Ka^\star\Ka] - \frac{1}{L^2}(\Pu-\Pul)\Ptl(\Pu-\Pul)\right\Vert \leq \frac{\nu r}{nL} + \frac{4}{L^2}\leq \frac{2\nu r}{nL},
\end{align*}
so $V_0 := \frac{4 \nu^2 r^2}{nL}$. Now, for $t<\frac{mV_0}{c} = \frac{2m}{L}$, the result follows from Theorem~\ref{thm:Bernstein} with $t = \sqrt{\frac{4\beta\nu^2 r^2 n\log(n)}{3m}}$ with $m\geq \frac{4}{3}\nu^2 r^2 n\log(n)$.
\end{proof}

\section{Proof of Local Convergence (Theorem~\ref{thm: Local Convergence})}\label{appendix: local convergence}
In this section, we will use the properties proven thus far to provide proof of local convergence of Algorithm~\ref{alg:R_omega descent}.

We begin with the following technical lemmas:
\begin{lem}[Stepsize Bounds]\label{lem:Stepsize}
Assume that $\Vert \Ptl - \frac{L}{m}\Ptl\Ro\Ptl\Vert\leq 4\varepsilon_0<1$. Then the stepsize $\alpha_l$ in Algorithm~\ref{alg:R_omega descent} can be bounded by
\begin{equation*}
    \frac{L/m}{1+4\varepsilon_0}\leq \alpha_l = \frac{\Vert \Pt\G_l\Vert_\fro^2}{\langle \Ptl\G_l,\Ro\Ptl\G_l\rangle}\leq \frac{L/m}{1-4\varepsilon_0}.
\end{equation*}
\end{lem}
\begin{proof}[Proof of Lemma~\ref{lem:Stepsize}]
We will prove this by leveraging the local RIP assumption. Notice the following:
\begin{align*}
        \langle \Ptl\G_l,\Ro\Ptl\G_l\rangle &= \langle \Ptl\G_l,\Ptl\Ro\Ptl\G_l\rangle\\
        & = \left\langle \Ptl\G_l,\Ptl\Ro\Ptl\G_l-\frac{m}{L}\Ptl\G_l \right\rangle + \frac{m}{L}\la\Pt\G_l,\Pt\G_l\ra.
\end{align*}
We can now leverage the variational characterization of the spectral norm and local RIP, proven in Lemma~\ref{lem: New Local Pseudo proof}, to bound the following:
\[
-\frac{m}{L}(4\varepsilon_0)\Vert \Pt\G_l\Vert_\fro^2 \leq \left\langle \Ptl\G_l,\Ptl\Ro\Ptl\G_l-\frac{m}{L}\Ptl\G_l \right\rangle \leq \frac{m}{L}(4\varepsilon_0)\Vert \Pt\G_l\Vert_\fro^2 .
\]
As such, we can now bound the denominator as 
\[
\frac{m}{L}(1-4\varepsilon_0)\Vert \Pt\G_l\Vert_\fro^2\leq\langle \Ptl\G_l,\Ro\Ptl\G_l\rangle \leq \frac{m}{L}(1+4\varepsilon_0)\Vert \Pt\G_l\Vert_\fro^2.
\]

Rearrangement of this last expression yields the upper and lower bounds on the step size derived above. The condition that $4\varepsilon_0<1$ is required to enforce the positivity of the step size, as negative step sizes cause divergence in the contractive sequence. This is necessary as $\Ro$ is not a self-adjoint positive semi-definite operator. This concludes the proof.
\end{proof}

\begin{lem}[$I_1$ Bound]\label{lem: Alpha spectral norm bound}
    Assume $\left\Vert \Ptl - \frac{L}{m}\Ptl\Ro\Ptl\right\Vert\leq 4\varepsilon_0$ and $\alpha_l$ can be bounded as in Lemma~\ref{lem:Stepsize}. Then the spectral norm of $\Ptl - \alpha_l\Ptl\Ro\Ptl$ can be bounded as
    \begin{equation}
        \Vert \Ptl-\alpha_l\Ptl\Ro\Ptl\Vert \leq \frac{8\varepsilon_0}{1-4\varepsilon_0}.
    \end{equation}
\end{lem}
\begin{proof}[Proof of Lemma~\ref{lem: Alpha spectral norm bound}]
    From direct calculation, it follows that
    \begin{align*}
        \Vert \Ptl - \alpha_l\Ptl\Ro\Ptl\Vert &\leq \left\Vert\Ptl - \frac{L}{m}\Ptl\Ro\Ptl\right\Vert + \left\vert \alpha_l -\frac{L}{m}\right\vert\Vert\Ptl \Ro\Ptl\Vert\\
        &\leq 4\varepsilon_0+ \left\vert \alpha_l -\frac{L}{m}\right\vert\left(\left\Vert\Ptl \Ro\Ptl -\frac{m}{L}\Ptl\right\Vert +\frac{m}{L}\Vert\Ptl\Vert \right)\\
        &\leq  4\varepsilon_0 + \left(\frac{L/m}{1-4\varepsilon_0} - \frac{L/m(1-4\varepsilon_0)}{1-4\varepsilon_0}\right)\left(4\varepsilon_0\frac{m}{L} + \frac{m}{L} \right)\\
        &\leq 4\varepsilon_0 +  \frac{4\varepsilon_0}{1-4\varepsilon_0}(1+4\varepsilon_0)\\
        &=\frac{8\varepsilon_0}{1-4\varepsilon_0}.
    \end{align*}
    This finishes the proof.
\end{proof}

We can now prove Theorem~\ref{thm: Local Convergence}.

\subsection{Proof of Theorem~\ref{thm: Local Convergence}}
\begin{proof}\label{proof: Local convergence}
    First, it follows that
    \begin{equation*}
        \Vert \X_{l+1} - \X\Vert_\fro \leq \Vert\X_{l+1} -\W_l\Vert_\fro + \Vert \W_l - \X\Vert_\fro\leq 2\Vert \W_l - \X\Vert_\fro,
    \end{equation*}
    as $\X_{l+1}$ is the best rank-$r$ approximation of $\W_l$. Plugging in $\W_l = \X_l + \alpha_l\Ptl\G_l$, we see that
    \begin{align*}
        \Vert \X_{l+1}-\X\Vert_\fro&\leq 2\left\Vert \X_l + \alpha_l\Ptl\G_l - \X\right\Vert_\fro\\
        &=2\Vert \X_l-\X -\alpha_l\Ptl\Ro(\X_l-\X)\Vert_\fro\\
        &\leq \underbrace{2\Vert (\Ptl - \alpha_l\Ptl\Ro\Ptl)(\X_l-\X)\Vert_\fro}_{I_1}\\
        &\quad+ \underbrace{2\Vert(I-\Ptl)(\X_l-\X)\Vert_\fro}_{I_2}\\
        &\quad+ \underbrace{2\vert\alpha_l\vert\Vert\Ptl\Ro(I-\Ptl)(\X_l-\X)\Vert_\fro}_{I_3}.
    \end{align*}
    It remains to bound each term individually. Using Lemma~\ref{lem: Alpha spectral norm bound}, we see that
    \begin{equation*}
        I_1\leq \frac{16\varepsilon_0}{1-4\varepsilon_0}\Vert\X_l-\X\Vert_\fro.
    \end{equation*}
    Next, notice that from Lemma~\ref{lem: Projection Bounds} and the fact that $\Ptl\X_l = \X_l$,
    \begin{align*}
        I_2 &= 2\Vert(I-\Ptl)\X_l-(I-\Ptl)\X)\Vert_\fro\\
        &= 2\Vert(I-\Ptl)\X\Vert_\fro\\
        &\leq\frac{2\Vert \X_l-\X\Vert_\fro^2}{\sigma_\text{min}(\X)}\\
        &\leq\frac{\sqrt{m}\varepsilon_0}{8n^{5/4}\sqrt{\beta\nu r\log{n}}} \Vert\X_l-\X\Vert_\fro\\
        &\leq \varepsilon_0\Vert \X_l-\X\Vert_\fro\\
        &\leq \frac{\varepsilon_0}{1-4\varepsilon_0}\Vert \X_l-\X\Vert_\fro,
    \end{align*}
    using Lemma~\ref{lem: Projection Bounds} and our initial assumption.
    Finally, we see that, following a similar argument as in the bound of $I_2$ and using Lemma~\ref{lem: Ptl R_o bound},
    \begin{align*}
        I_3&\leq 2\vert\alpha_l\vert\Vert \Ptl\Ro\Vert\Vert(I-\Ptl)\X\Vert_\fro\\
        &\leq \frac{2L/m}{1-4\varepsilon_0}\left[\left(\frac{m}{L} + 4\sqrt{\frac{8m\log(n)}{n}}\right)\frac{2\Vert \X_l - \X\Vert_\fro}{\sigma_\mathrm{min}(\X)} + \frac{m}{L}+\frac{4m\sqrt{n}}{L}\sqrt{\frac{\beta \nu r n\log(n)}{3m}} \right]\left(\frac{\Vert\X_l-\X\Vert_\fro}{\sigma_\mathrm{min}(\X)}\right)\Vert\X_l-\X\Vert_\fro\\
        &\leq\frac{1}{1-4\varepsilon_0}\left(\frac{\varepsilon_0^2}{128} + \frac{\varepsilon_0}{16}+\frac{\varepsilon_0^2}{32}+\frac{\varepsilon_0}{\sqrt{48}}\right)\Vert\X_l-\X\Vert_\fro\\
        &\leq \frac{\varepsilon_0}{1-4\varepsilon_0},
    \end{align*}
    where the second to last inequality follows from the same analysis conducted in Lemma~\ref{lem: New Local Pseudo proof}, just divided by 2. Collecting these results, we get
    \begin{equation*}
        \Vert\X_{l+1}-\X\Vert_\fro\leq\frac{18\varepsilon_0}{1-4\varepsilon_0}\Vert\X_l-\X\Vert_\fro.
    \end{equation*}
    By the assumption of the theorem, which holds for $l=0$, and as we have a contractive sequence, it inductively follows that the assumption holds for $l\geq 0$. This concludes the proof.
\end{proof}

\section{Initialization Results}\label{appendix: initialization}
Now that local convergence has been established, we can now prove quantitative guarantees for the initialization methods provided in the main text.

\subsection{Proof of Lemma~\ref{lem:Initialization}}
We first start with a proof of the guarantee provided by one-step hard thresholding, detailed in Lemma~\ref{lem:Initialization}:
\begin{proof}\label{proof: 1SHT}
     First, notice that for $\W_0 = \frac{L}{m}\Ro(\X)$, we get
    \begin{align*}
        \left\Vert \X_0 - \X\right\Vert&\leq \left\Vert \W_0 -\X\right\Vert + \left\Vert \W_0-\X_0\right\Vert\\
        &\leq 2\left\Vert \W_0-\X\right\Vert,
    \end{align*}
    where the first inequality follows from the triangle inequality and the second inequality follows from the fact that $\W_0$ is the best rank-$r$ approximation of $\X_0$ by Eckart-Young-Mirsky\cite{eckart1936approximation}. We now need a bound for this last term. Notice that $\W_0-\X = \frac{L}{m}\sum_{\alphab}\langle \X,\wa\rangle \va - \X$ is a sum of zero-mean i.i.d random matrices, opening up use of Bernstein's inequality. In order to use this, define $\Z_{\alphab} = \frac{L}{m}\langle \X,\wa\rangle\va -\frac{1}{m}\X$. We need a bound on $\left\Vert \Z_{\alphab}\right\Vert$ and $\left\Vert \mathbb{E}[\Z_{\alphab}]^2\right\Vert$.
    First, notice that 
    \begin{align*}
        \left\Vert \Z_{\alphab} \right\Vert &= \left\Vert \frac{L}{m}\langle \X,\wa\rangle\va-\frac{1}{m}\X\right\Vert\\
        &\leq\frac{L}{m}\vert\langle \X,\wa\rangle\vert \left\Vert \va\right\Vert + \frac{1}{m}\left\Vert \X\right\Vert\\
        &\leq \frac{4L}{m}\left\Vert \X\right\Vert_\infty + \frac{n}{m}\left\Vert \X\right\Vert_\infty\\
        &\leq \frac{5L}{m}\left\Vert\X\right\Vert_\infty =:c,
    \end{align*}
    as $\left\Vert \va\right\Vert<1$ from Lemma~\ref{lem: H and H^-1 eigvals}.
    Next, notice that 
    \begin{equation*}
        \left\Vert\mathbb{E}[\Z_{\alphab}^2]\right\Vert = \left\Vert\frac{L}{m^2}\sum_{\alphab} \langle \X,\wa\rangle^2\va^2-\frac{1}{m^2}\X^2\right\Vert.
    \end{equation*}
    As both these matrices are positive semi-definite, it follows that $\left\Vert\mathbb{E}[\Z_{\alphab}^2]\right\Vert\leq\max\{ \left\Vert \frac{L}{m^2}\sum_{\alphab} \langle \X,\wa\rangle^2\va^2\right\Vert,\left\Vert \frac{1}{m^2}\X^2\right\Vert\}$.
    It follows that
    \begin{align*}
        \left\Vert \frac{L}{m^2} \sum_{\alphab} \langle \X,\wa\rangle^2\va^2\right\Vert &= \max_{\y\in\bb{R}^n,\left\Vert \y\right\Vert_2 = 1} \frac{L}{m^2} \sum_{\alphab} \langle \X,\wa\rangle^2 \y^\top\va^2\y\\
        &\leq \frac{16L}{m^2}\left\Vert \X\right\Vert_\infty^2 \y^\top\left(\sum_{\alphab} \va^2\right)\y\\
        &= \frac{16L}{m^2}\left\Vert \X\right\Vert_\infty^2 \lambda_\text{max}\left(\sum_{\alphab}\va^2\right).
    \end{align*}
    Now, from Lemma~\ref{lem:Form of va^2}, $\sum_{\alphab} \va^2 = \frac{n^2 - 2n +2}{4n}\J$. It follows that $\lambda_\text{max}\left(\sum_{\alphab} \va^2 \right) = \frac{n^2-2n+2}{4n}\leq \frac{n}{4}$ as $\J$ is an orthogonal projection. Thus,
    \begin{equation*}
        \left\Vert \mathbb{E}\left(\Z_{\alphab}^2\right)\right\Vert \leq \frac{5nL}{m^2}\left\Vert \X\right\Vert^2_\infty =: V_0.
    \end{equation*}
    Now to determine $t$, we note that
    \begin{align*}
        \frac{V}{c} &= \frac{5nL}{m}\left\Vert \X\right\Vert_\infty^2\frac{m}{5L\left\Vert\X\right\Vert_\infty}\\
        &= n\left\Vert\X\right\Vert_\infty\\
        &\geq \sqrt{\frac{40\beta n^3\log{n}}{3m}}\Vert\X\Vert_\infty,
    \end{align*}
    for $m\geq \frac{40}{3}\beta n\log{n}$. It follows that
    \begin{align*}
        \bb{P}\left(\left\Vert\X_0-\X\right\Vert>\sqrt{\frac{40\beta n^3\log{n}}{3m}}\Vert \X\Vert_\infty\right)&\leq 2n \exp\left(-\beta\log(n)\right)\\
        &= 2n^{1-\beta},
    \end{align*}
    verifying the probabilistic bound. To complete the proof we use Assumption~\ref{asp: mu_1 assumption}, and it follows that
    \begin{equation*}
        \Vert \X_0-\X\Vert_\fro\leq\sqrt{2r}\Vert\X_0-\X\Vert \leq \sqrt{\frac{320 r\beta n^3\log(n)}{3m}}\Vert\X\Vert_\infty \leq \sqrt{\frac{320 r^2 \mu_1^2 n\log(n)}{3m}}\Vert \X\Vert,
    \end{equation*}
    thus concluding the proof.
\end{proof}

Now, we will prove a technical lemma about Algorithm~\ref{alg:Trimming Algorithm}:
\begin{lem}[Trimming Result] \label{lem: Trimming result}
Let $\Z_l = \U_l\D_l\U_l^\top$ be a rank-$r$ matrix such that
\begin{equation*}
    \Vert \Z_l - \X\Vert\leq\frac{\sigma_\mathrm{min}(\X)}{10\sqrt{2}}.
\end{equation*}
Then the matrix $\hat{\Z_l}$ returned by Algorithm~\ref{alg:Trimming Algorithm} satisfies
\begin{equation*}
    \Vert \mathcal{P}_{\hat{\U}_l} \e_i\Vert \leq \frac{10}{9}\sqrt{\frac{\nu r}{128n}},\qquad\qquad \Vert \mathcal{P}_{\hat{\U}_l}\wa\Vert_\fro \leq \frac{10}{9} \sqrt{\frac{\nu r}{8n}}, \qquad \mathrm{and}\qquad \Vert \mathcal{P}_{\hat{\U}_l}\va\Vert_\fro\leq \frac{10}{9} \sqrt{\frac{\nu r}{2n}}
\end{equation*}
and furthermore
\begin{equation*}
    \Vert\hat{\Z}_l-\X\Vert_\fro\leq 8\kappa\Vert\Z_l-\X\Vert_\fro
\end{equation*}
\end{lem}
\begin{proof}[Proof of Lemma~\ref{lem: Trimming result}]
    The proof of the first and fourth statements can be found in \cite{wei2020guarantees}. To see the second and third statements, we can apply the same analysis as in Remark~\ref{rem: incoherence remark}. This analysis is reproduced here for convenience. First, notice that if
    $\Vert \mathcal{P}_{\hat{\U}_l}\e_{ij}\Vert_2\leq\frac{10}{9}\sqrt{\frac{\nu r}{128n}}$, then by the triangle inequality
    \[
    \Vert \mathcal{P}_{\hat{\U}_l}\wa\Vert_\fro^2 \leq \frac{40}{9}\Vert\Pu\e_{ij}\Vert_\fro\leq \frac{10}{9}\sqrt{\frac{\nu r}{8n}}.
    \]
    This validates the second statement. To see the third result, notice that
        \begin{align*}
        \Vert\mathcal{P}_{\hat{\U}_l}\va\Vert_\fro &= \left\Vert \mathcal{P}_{\hat{\U}_l}\left(\sum_{\betab\in\mathbb{I}}H^{\alphab\betab}\wb\right)\right\Vert_\fro\\
        &\leq\sum_{\betab\in\mathbb{I}}\left\vert H^{\alphab\betab}\right\vert\Vert\mathcal{P}_{\hat{\U}_l}\wb\Vert_\fro\\
        &\leq \frac{10}{9}\sqrt{\frac{\nu r}{8n}}\sum_{\betab\in\mathbb{I}}\left\vert H^{\alphab \betab}\right\vert,
    \end{align*}
    and from Lemma~\ref{lem: H and H^-1 eigvals}, it follows that
        $\sum_{\ai}\vert H^{\alphab\betab}\vert \leq 2$, so the last statement follows.
\end{proof}

\subsection{Proof of Lemma~\ref{lem:Resampling result}}
\begin{proof}\label{proof: Resampling proof}
    First, assume that at the $l$-th iteration of Algorithm~\ref{alg:Resampling Algorithm},
    \begin{equation*}
        \Vert\Z_l-\X\Vert_\fro\leq\frac{\sigma_\mathrm{max}(\X)}{256\kappa^2}.
    \end{equation*}
    This indicates that $\hat{\Z_l}$ is $\frac{100}{81}$-$\nu$ incoherent with respect to $\{\wa\}_{\ai}$ and that
    \begin{equation*}
        \Vert \hat{\Z}_l - \X\Vert_\fro \leq 8\kappa \Vert\Z_l-\X\Vert_\fro.
    \end{equation*}

    Following a similar strategy as in the proof for Theorem~\ref{thm: Local Convergence}, we can decompose the error at the $l+1$-th iteration as follows:
    \begin{align*}
        \Vert \Z_{l+1} -\X\Vert_\fro &\leq \underbrace{2\left\Vert (\Pthl - \frac{L}{m}\Pthl \Rol \Pthl)(\hat{\Z}_l-\X)\right\Vert_\fro}_{I_4}\\
        &\quad+ \underbrace{2\left\Vert (\mathcal{I}-\Pthl)(\hat{\Z}_l-\X)\right\Vert_\fro}_{I_5}\\
        &\quad+ \underbrace{2\left\Vert \frac{L}{m}\Pthl \Rol(\mathcal{I}-\Pthl)(\hat{\Z}_l-\X)\right\Vert_\fro}_{I_6}.
    \end{align*}
    Now, as $\hat{\Z}_l$ and $\Omega_{l+1}$ are independent, we can use Theorem~\ref{thm: R_omega RIP} to bound $I_4$ as follows:
    \begin{align*}
        I_4&\leq 2\left\Vert \frac{L}{m}\Pthl\Rol\Pthl - \Pthl\right\Vert\Vert \hat{\Z}_l-\X\Vert_\fro\\
        &\leq 2\sqrt{\frac{8(100\nu/81 )^2r^2\beta n\log(n)}{3\hat{m}}}\Vert \hat{\Z}_l-\X\Vert_\fro\\
        &\leq 16\kappa \sqrt{\frac{80000 \nu^2r^2\beta n\log(n)}{19683\hat{m}}}\Vert \Z_l-\X\Vert_\fro\\
        &\leq 32\kappa\sqrt{\frac{ \nu^2r^2\beta n\log(n)}{\hat{m}}}\Vert \Z_l-\X\Vert_\fro\label{eqn: I_4 bound},
    \end{align*}
    with probability at least $1-2n^{1-\beta}$ as long as $\hat{m}\geq 4\nu^2 r^2\beta n\log(n)$.

    Next, we can bound $I_5$ as follows:
    \begin{align*}
     I_5 &\leq \frac{2\Vert\hat{\Z}_l-\X\Vert_\fro^2}{\sigma_\mathrm{min}(\X)}\\
     &\leq \frac{128\kappa^2\Vert \Z_l-\X\Vert_\fro^2}{\sigma_\mathrm{min}(\X)}\\
     &\leq \frac{1}{2}\Vert\Z_l-\X\Vert_\fro,
    \end{align*}
    where the first inequality follows from Lemma~\ref{lem: Projection Bounds}, the second inequality comes from Lemma~\ref{lem: Trimming result}, and the last inequality comes from the starting assumption.

    For the final result, again recall that $\hat{\Z}_l$ and $\Omega_{l+1}$ are independent, and $\hat{\Z}_l$ is $\frac{100}{81}$-$\nu$ incoherent with respect to $\{\wa\}_{\ai}$. 
    From Lemma~\ref{lem: Asymmetric RIP} and the new incoherence parameter, we have that 
    \begin{align*}
        \left\Vert \frac{L}{m}\Pthl \Rol\Pthl (\Pu - \Puhl) - \Pthl(\Pu-\Puhl)\right\Vert&\leq \sqrt{\frac{40000\beta\nu^2 r^2 n\log(n)}{19683\hat{m}}}\\
        &\leq \sqrt{\frac{2\beta\nu^2 r^2 n\log(n)}{\hat{m}}},
    \end{align*}
    with probability at least $1-2n^{1-\beta}$ given $\hat{m}\geq 2\beta\nu^2 r^2 n\log(n)$.
    Now, as
    \begin{align*}
        (\mathcal{I}-\Pthl)(\hat{\Z}_l-\X) &= -(\mathcal{I}-\Pthl)(\X)\\
        &= -\X + \hat{\U}_l\hat{\U}_l^\top \X + \X\hat{\U}_l\hat{\U}_l^\top - \hat{\U}_l\hat{\U}_l^\top\X\hat{\U}_l\hat{\U}_l^\top\\
        &=-\U\U^\top\X + \hat{\U}_l\hat{\U}_l^\top\X + \U\U^\top\X\hat{\U}_l\hat{\U}_l^\top-\hat{\U}_l\hat{\U}_l^\top\X\hat{\U}_l\hat{\U}_l^\top\\
        &= -(\U\U^\top-\hat{\U}_l\hat{\U}_l^\top)\X(\I-\hat{\U}_l\hat{\U}_l^\top)\\
        &=(\U\U^\top-\hat{\U}_l\hat{\U}_l^\top)(\hat{\Z}_l-\X)(\I-\hat{\U}_l\hat{\U}_l^\top)\\
        &=(\Pu-\Puhl)(\hat{\Z}_l-\X)(\mathcal{I}-\Puhl),
    \end{align*}
    where the first line follows from the fact that $\Pthl \hat{\Z}_l = \hat{\Z}_l$, the second line follows from the definition of $\Pthl$, the third line follows from the fact that $\U\U^\top \X = \X$, the fourth line is a rearrangement of terms, and the fifth line follows from the fact that $\hat{\Z}_l(\I-\hat{\U}_l\hat{\U}_l^\top) = \bm{0}$. It follows that
    \begin{align*}
        I_6 &= 2\left\Vert\frac{L}{m}\Pthl \Rol(\Pu - \Puhl)(\hat{\Z}_l-\X)(\mathcal{I}-\Puhl)\right\Vert_\fro\\
        &=2 \left\Vert\frac{L}{m}\Pthl\Rol(\Pu-\Puhl)(\hat{\Z}_l-\X)(\mathcal{I}-\Puhl)-\underbrace{\Pthl(\Pu-\Puhl)(\hat{\Z}_l-\X)(\mathcal{I}-\Puhl)}_{=\bm{0}}\right\Vert_\fro\\
        & \leq 2\left\Vert\frac{L}{m}\Pthl\Rol(\Pu-\Puhl)-\Pthl(\Pu-\Puhl)\right\Vert\left\Vert(\hat{\Z}_l-\X)(\mathcal{I}-\Puhl)\right\Vert_\fro\\
        &\leq 2\left\Vert\frac{L}{m}\Pthl\Rol(\Pu-\Puhl)-\Pthl(\Pu-\Puhl)\right\Vert\left\Vert\hat{\Z}_l-\X\right\Vert_\fro\\
        &\leq 16\kappa\sqrt{\frac{2\beta\nu^2 r^2 n\log(n)}{\hat{m}}},
    \end{align*}
    where the first two lines follow from the computation above. Combining $I_4, ~I_5$, and $I_6$ gives 
    \begin{equation} \label{eqn: Resampling error}
        \Vert \Z_{l+1}-\X\Vert_\fro\leq \left(\frac{1}{2}+64\kappa\sqrt{\frac{\beta\nu^2 r^2 n\log(n)}{\hat{m}}}\right)\Vert\Z_l-\X\Vert_\fro,
    \end{equation}
    with probability at least $1-4n^{1-\beta}$. It follows that \eqref{eqn: Resampling error} is less than $\frac{5}{6}$ for $\hat{m}\geq (192\nu r\kappa)^2\beta n\log(n)$.

    Now, as $\Z_0 = \mathcal{H}_r\left(\frac{L}{\hat{m}}\mathcal{R}_{\Omega_0}(\X) \right)$, we can make $\Vert \Z_0-\X\Vert_\fro\leq \frac{\sigma_\mathrm{min}(\X)}{256\kappa^2}$ using the one step hard thresholding result from Lemma~\ref{lem:Initialization} for $\hat{m}\geq (2\times10^5)\kappa^6 r^2\mu_1^2 n\log(n)$, and the result follows from here. No attempts were made to optimize the constants.
\end{proof}

\section{Miscellaneous Results}\label{appendix: misc}

\begin{lem}[Bounds for Projections]\label{lem: Projection Bounds}
    Let $\X_l = \U_l \D_l \U_l^\top$ be a rank-$r$ matrix and $\T_l$ be the tangent space of the rank-$r$ matrix manifold at $\X_l$. Let $\X = \U\D\U^\top$ be another rank-$r$ matrix, and $\T$ be the corresponding tangent space. Then 
    \begin{gather*}
        \Vert \U_l\U_l^\top - \U\U^\top\Vert\leq\frac{\Vert \X_l-\X\Vert_\fro}{\sigma_\mathrm{min}(\X)},\qquad \qquad
        \Vert \U_l\U_l^\top - \U\U^\top\Vert_\fro \leq \frac{\sqrt{2}\Vert\X_l-\X\Vert_\fro}{\sigma_\mathrm{min}(\X)}\\
        \Vert (\mathcal{I} - \Ptl)\X\Vert_\fro \leq \frac{\Vert \X_l-\X\Vert_\fro^2}{\sigma_\mathrm{min}(\X)},\qquad \qquad
        \Vert\Ptl-\Pt\Vert\leq\frac{2\Vert \X_l-\X\Vert_\fro}{\sigma_\mathrm{min}(\X)}.
    \end{gather*}
\end{lem}
\begin{proof}[Proof of Lemma~\ref{lem: Projection Bounds}]
    See \cite{Weirecovery2016,wei2020guarantees}.
\end{proof}

\end{document}